\documentclass[reqno, 11pt, centertags,draft]{amsart}
\usepackage{amsmath,amsthm,amscd,amssymb,latexsym,upref,stmaryrd}
\usepackage{color,soul}

\newtheorem{theorem}{Theorem}[section]
\newtheorem{proposition}[theorem]{Proposition}
\newtheorem{corollary}[theorem]{Corollary}
\newtheorem{lemma}[theorem]{Lemma}
\theoremstyle{definition}
\newtheorem{definition}[theorem]{Definition}
\theoremstyle{remark}
\newtheorem{remark}[theorem]{Remark}
\newtheorem{example}[theorem]{Example}

\numberwithin{equation}{section}

\DeclareMathOperator{\col}{col}
\DeclareMathOperator{\const}{const}
\DeclareMathOperator{\diag}{diag}
\DeclareMathOperator{\dom}{dom}
\DeclareMathOperator{\ran}{ran}
\DeclareMathOperator{\rank}{rank}
\DeclareMathOperator{\Span}{span}
\DeclareMathOperator{\supp}{supp}
\let\Re\relax
\DeclareMathOperator{\Re}{Re}
\let\Im\relax
\DeclareMathOperator{\Im}{Im}
\DeclareMathOperator{\Ext}{Ext}

\newcommand{\wt}{\widetilde}
\newcommand{\wh}{\widehat}
\newcommand{\ol}{\overline}
\renewcommand{\(}{\left(}
\renewcommand{\)}{\right)}
\renewcommand{\le}{\leqslant}

\newcommand{\eps}{\varepsilon}
\renewcommand{\l}{\lambda}

\def\cB{\mathcal{B}}

\def\cL{\mathcal{L}}
\def\cR{\mathcal{R}}

\def\fH{\mathfrak{H}}
\def\fS{\mathfrak{S}}

\def\bC{\mathbb{C}}
\def\bD{\mathbb{D}}
\def\bN{\mathbb{N}}
\def\bR{\mathbb{R}}
\def\bZ{\mathbb{Z}}

\begin{document}

\sloppy

\title[Completeness property of perturbations of normal and spectral BVP]
{Completeness property \\
of one-dimensional perturbations \\
of normal and spectral operators \\ 
generated by first order systems}

\author[A.~V.~Agibalova]{Anna~V.~Agibalova}

\address{%
Donetsk National University \\
Universitetskaya str. 24 \\
83055 Donetsk \\
Ukraine}

\email{agannette@rambler.ru}

\author[A.~A.~Lunyov]{Anton~A.~Lunyov}

\address{%
Facebook, Inc., MPK 12 \\
12 Hacker Way \\
Menlo Park, California, 94025 \\
United States of America}

\email{A.A.Lunyov@gmail.com}

\author[M.~M.~Malamud]{Mark~M.~Malamud}

\address{%
Peoples Friendship University of Russia (RUDN University) \\
6 Miklukho-Maklaya St. \\
Moscow, 117198, \\
Russian Federation}

\email{mmm@telenet.dn.ua}

\author[L.~L.~Oridoroga]{Leonid~L.~Oridoroga}

\address{%
Donetsk National University \\
Universitetskaya str. 24 \\
83055 Donetsk \\
Ukraine}

\email{vremenny-orid@mail.ru}

\subjclass{Primary 47E05; Secondary 34L10, 47B15}

\keywords{Systems of ordinary differential equations, normal operator,
completeness of root vectors, resolvent operator, rank one perturbation, Riesz
basis property}

\begin{abstract}
The paper is concerned with completeness property of rank one perturbations of
unperturbed operators generated by special boundary value problems (BVP) for the
following $2 \times 2$ system
\begin{equation} \label{eq:Ly.abstract}
 L y = -i B^{-1} y' + Q(x) y = \lambda y , \quad
 B = \begin{pmatrix} b_1 & 0 \\ 0 & b_2 \end{pmatrix}, \quad
 y = \begin{pmatrix} y_1 \\ y_2 \end{pmatrix},
\end{equation}
on a finite interval assuming that a potential matrix $Q$ is summable, and $b_1
b_2^{-1} \notin \bR$ (essentially non-Dirac type case). We assume that
unperturbed operator generated by a BVP belongs to one of the following three
subclasses of the class of spectral operators:

(a) normal operators;

(b) operators similar either to a normal or almost normal;

(c) operators that meet Riesz basis property with parentheses;

\noindent We show that in each of the three cases there exists (in general,
non-unique) operator generated by a quasi-periodic BVP and its certain rank-one
perturbations (in the resolvent sense) generated by special BVPs which are
complete while their adjoint are not.

In connection with the case (b) we investigate Riesz basis property of
quasi-periodic BVP under certain assumptions on a potential matrix $Q$. We also
find a simple formula for the rank of the resolvent difference for operators
corresponding to two BVPs for $n \times n$ system in terms of the coefficients
of boundary linear forms.
\end{abstract}

\maketitle{}

\tableofcontents

\section{Introduction} \label{sec:intro}
%
%
\textbf{1.1.} During the last two decades there appeared numerous papers devoted
to completeness and Riesz basis properties in $L^2([0,1]; \bC^{n})$ of boundary
value problems (BVP) for general first order system of ODE
\begin{equation} \label{eq:system.nxn}
 \cL y := \cL(B, Q)y := -i B^{-1} y' + Q(x)y = \l y,
 \quad y = \col(y_1,...,y_n).
\end{equation}
Here $B$ is a nonsingular diagonal $n\times n$ matrix with complex
entries, $B = \diag(b_1, b_2, \ldots, b_n) \in \bC^{n\times n}$, and
$Q(\cdot) =: (Q_{jk}(\cdot))_{j,k=1}^n \in L^1([0,1]; \bC^{n\times
n})$ is a potential matrix.

To obtain a BVP, equation~\eqref{eq:system.nxn} is subject to the
following boundary conditions (BC)
\begin{equation} \label{eq:BC.nxn}
 C y(0) + D y(1) = 0, \quad C= (c_{jk}), \ D = (d_{jk}) \in \bC^{n \times n}.
\end{equation}
We always impose the maximality condition $\rank(C \ D) = n$.

With system~\eqref{eq:system.nxn} one associates, in a natural way, the maximal
operator $L_{\max} := L_{\max}(B,Q)$ acting in $L^2([0,1]; \bC^n)$ on the domain
\begin{equation*}
 \dom(L_{\max}) = \{y \in AC([0,1]; \bC^n) : \cL y \in L^2\([0,1]; \bC^n\)\}.
\end{equation*}
Clearly, $\dom(L_{\max}) = W^{1,2}\([0,1]; \bC^n\)$ whenever $Q(\cdot) \in
L^2([0,1]; \bC^{n \times n})$. In this case the minimal operator $L_{\min} :=
L_{\min}(B,Q)$ is a restriction of $L_{\max}$ to
\begin{equation*}
 \dom(L_{\min}) = W^{1,2}_0([0,1]; \bC^n) :=
 \{y\in W^{1,2}([0,1]; \bC^n): y(0) = y(1) = 0\}.
\end{equation*}
Denote by $L_{C,D} := L_{C,D}(B,Q)$ the operator associated in $L^2([0,1]; \bC^n)$
with the BVP~\eqref{eq:system.nxn}--\eqref{eq:BC.nxn}. It is defined as the
restriction of $L_{\max}(B,Q)$ to the set of functions satisfying~\eqref{eq:BC.nxn}.

Apparently, the spectral problems~\eqref{eq:system.nxn}--\eqref{eq:BC.nxn}
have first been investigated by G.~D.~Birkhoff and R.~E.~Langer~\cite{BirLan23}.
Namely, they have extended certain previous results due to~Birkhoff and Tamarkin
on non-selfadjoint BVP for ODE to the case of BVP~\eqref{eq:system.nxn}--\eqref{eq:BC.nxn}.
More precisely, they introduced the concepts of \emph{regular and strictly
regular boundary conditions} and investigated the asymptotic behavior of
eigenvalues and eigenfunctions of the corresponding operator $L_{C,D}(B,Q)$
assuming that a potential matrix $Q(\cdot)$ is continuous. Moreover, they
proved \emph{a pointwise convergence result} on spectral decompositions of
the operator $L_{C,D}(B,Q)$ corresponding to the
BVP~\eqref{eq:system.nxn}--\eqref{eq:BC.nxn} with regular boundary conditions.

The completeness property of the root vectors system \emph{of general BVP} for
equation~\eqref{eq:system.nxn} has first been investigated in the recent
paper~\cite{MalOri12}. In this paper the concept of \emph{weakly regular} boundary
conditions~\eqref{eq:BC.nxn} for the system~\eqref{eq:system.nxn} was introduced
and the completeness of the root vectors for such type a BVP was proved
(see Theorem~\ref{th2.1_MO} in Appendix).

In the recent paper~\cite{LunMal15} it was established the Riesz basis property
with parentheses for system~\eqref{eq:system.nxn} subject to various classes of
boundary conditions with a potential $Q(\cdot) \in L^\infty([0,1]; \bC^{n\times n})$.

\textbf{1.2.} Going over to the case $n=2$ we consider the system
\begin{equation}\label{eq:system}
 -i B^{-1} y' + Q(x)y = \l y, \qquad y = \col(y_1,y_2), \qquad x\in[0,1],
\end{equation}
with nonsingular matrix $B$ and complex valued potential matrix $Q,$
\begin{equation}\label{eq:BQ}
	B = \diag(b_1, b_2), \quad\text{and}\quad
	Q = \begin{pmatrix} 0 & Q_{12} \\ Q_{21} & 0 \end{pmatrix} \in
	L^1\([0,1]; \bC^{2 \times 2}\).
\end{equation}
In this case it is more convenient to rewrite conditions~\eqref{eq:BC.nxn} as
\begin{equation}\label{eq:BC}
	U_j(y) := a_{j 1}y_1(0) + a_{j 2}y_2(0) + a_{j 3}y_1(1) + a_{j 4}y_2(1)= 0,
	\quad j \in \{1,2\},
\end{equation}
where the linear forms $\{U_j\}_{j=1}^2$ are assumed to be linearly
independent. We also write $L_{U_1, U_2}$ instead of $L_{C, D}$.

As opposed to general problem~\eqref{eq:system.nxn}--\eqref{eq:BC.nxn},
BVP~\eqref{eq:system}--\eqref{eq:BC} with $B = \diag(-1, 1)$ (Dirac system) has
been investigated in great detail. First we mention that completeness property
of irregular and even degenerate BVP ~\eqref{eq:system.nxn}--\eqref{eq:BC.nxn}
was investigated in~\cite{MalOri12},~\cite{LunMal13Dokl}. Besides, P.~Djakov and
B.~Mityagin~\cite{DjaMit12Equi} imposing certain smoothness condition on
$Q(\cdot)$ proved equiconvergence of the spectral decompositions for $2 \times
2$ Dirac equations subject to \emph{general regular boundary conditions}.
Moreover, the Riesz basis property for $2\times 2$ Dirac operators $L_{U_1, U_2}$
has been investigated in numerous papers (see~\cite{TroYam02, DjaMit10BariDir,
Bask11, DjaMit12UncDir, DjaMit12Crit, DjaMit13CritDir, Gub03, LunMal14Dokl,
LunMal16JMAA, SavShk14} and references therein, and discussion in
Remark~\ref{rem:Riesz_bas_prop_for_Dirac}).

\textbf{1.3.} In this paper considering the case of $n=2$ we always assume that
\begin{equation}\label{eq:condition_on_B}
	B = \diag(b_1, b_2) \quad\text{and}\quad b_1 b_2^{-1} \not \in \bR.
\end{equation}
To describe the main aim of the paper we introduce the following definition.
%
%
\begin{definition} \label{def:peculiar}
\begin{itemize}
\item[(i)] An operator $S$ with discrete spectrum in a Hilbert space $\fH$ is
called \textbf{complete} if the system of its root vectors is complete in $\fH$;
\item[(ii)] We call an operator $S$ \textbf{peculiarly complete} if $S$ is
complete while the adjoint operator $S^*$ is not and the span of its root
vectors has infinite codimension in $\fH$.
\end{itemize}
\end{definition}
%
%
\begin{definition} \label{def:peculiar_pair}
A pair of operators $\{T,\wt{S}\}$ will be called \textbf{peculiar} if:
(a)~$T$~is~normal;
(b)~$\wt{S}$ is peculiarly complete; and
(c)~the resolvent difference $(\wt{S}-\l)^{-1} - (T-\l)^{-1}$ is finite-dimensional.
\end{definition}
%
%
Our main purpose here is to describe all peculiar pairs of operators $\{T :=
L_{U_1, U_2}(B, Q), \wt{S} := L_{\wt{U}_1,\wt{U}_2}(B, Q)\}$ provided that $B$
satisfies condition~\eqref{eq:condition_on_B}. Surprisingly such pairs exist only
in the trivial case of zero potential $Q \equiv 0$. We also find explicit conditions
in terms of coefficient $a_{jk}$ of the forms~\eqref{eq:BC} ensuring that the
resolvent difference $(\wt{S} - \l)^{-1} - (T - \l)^{-1}$ is one-dimensional.

To state the main result we need one more definition.
%
%
\begin{definition} \label{def:bc.equiv}
We call a pair of BC $U_1(y) = U_2(y) = 0$ equivalent to a pair of BC $V_1(y) =
V_2(y) = 0$, if they can be transformed to each other by means of simplest linear
transforms $i_1: \binom{y_1}{y_2} \mapsto \binom{y_2}{y_1}$ and $i_2: y(x) \mapsto
y(1-x)$.
\end{definition}
%
%
With this definition our main result reads as follows.
%
%
\begin{theorem} \label{th:one.dim.perturb.normal}
Let $n=2$ and let $T := L_{U_1, U_2}(B, Q)$ and $\wt{S} := L_{\wt{U}_1,\wt{U}_2}
(B, Q)$. A pair of operators $\{T, \wt{S}\}$ is peculiar, i.e. $T$ is normal and
$\wt{S}$ is peculiarly complete, if and only if $Q \equiv 0$ and pairs of boundary
conditions $\{U_1,U_2\}$, $\{\wt{U}_1,\wt{U}_2\}$ are equivalent, respectively,
to pairs $\{V_1,V_2\}$ and $\{\wt{V}_1,\wt{V}_2\}$, given by
\begin{align}
\label{eq:U1.U2.d1.d2.intro}
	V_1(y) &= y_1(0) - d_1 y_1(1) = 0, \qquad
	V_2(y) = y_2(0) - d_2 y_2(1) = 0, \\
\label{eq:wt.U1.U2.d1.h1.intro}
	\wt{V}_1(y) &= y_1(0) - h_1 y_2(0) = 0, \qquad
	\wt{V}_2(y) = y_1(1) - h_2 y_2(0) = 0,
\end{align}
where $d_j, h_j \in \bC$,\ $|d_1| = |d_2| = 1$ and $h_1h_2 \ne 0$.

Moreover, for such a pair of operators $\{T, \wt{S}\}$ the resolvent difference
$(\wt{S}-\l)^{-1} - (T-\l)^{-1}$ is one-dimensional if and only if $h_1 = d_1 h_2$.
\end{theorem}
%
%
Emphasize that our interest in this problem has been influenced by a recent
remarkable result by A. Baranov and D. Yakubovich~\cite{BarYak15, BarYak16}, which
we reformulate for unbounded operators with account of Definition~\ref{def:peculiar}.
%
%
\begin{theorem}~\cite{BarYak15,BarYak16,Bar18} \label{th:BarYak}
For any normal operator $L_0$ in $\fH$ with simple point spectrum there exists
peculiarly complete operator $L$ 
such that the resolvent difference $(L-\l)^{-1} - (L_0-\l)^{-1}$ is one-dimensional.
%
%
\end{theorem}
%
%
In fact, this result was proved in~\cite{BarYak15,BarYak16} only for $L_0 = L_0^*$
and was extended to the case of normal operators $L_0$ in a recent preprint by
A.~Baranov~\cite{Bar18}.

Note in this connection that the first (highly nontrivial) example of a peculiarly
complete operator $L$ (with selfadjoint $L_0$) was constructed by Hamburger
\cite{Ham51}. Later on Deckard, Foias and Pearcy~\cite{DecFoPea79} found a simpler
construction. However, in these examples the resolvent of the corresponding operator
$L$ is \emph{an infinite dimensional} perturbation of a selfadjoint compact operator
$(L_0 - \l)^{-1}$. Surprisingly, that in accordance with Theorem~\ref{th:BarYak}
one can find such examples among rank one perturbations.

Theorem~\ref{th:BarYak} substantially complements the classical Keldysh result
on completeness of weak perturbations of a selfadjoint finite order compact
operator (cf.~\cite{Kel51,Kel71,Shk16}). It is convenient to present its
"unbounded version".
%
%
\begin{theorem}~\cite[Theorem 5.10.1]{GohKre65} \label{Keldysh_th}
Let $L_0$ be a selfadjoint operator in $\fH$ with discrete spectrum and let $K$
be an $L_0$-compact operator such that $L_0^{-1} K L_0^{-1} \in \fS_p$ for some
$p \in (0,\infty)$. Then the operator $L = L_0 + K$ has discrete spectrum and is
complete. Moreover, the adjoint operator $L^*$ is also complete.
%
%
\end{theorem}
%
%
Note, that under the assumptions of Theorem~\ref{Keldysh_th} we have $\dom L =
\dom L_0$, meaning that $L$ is an additive perturbation of $L_0$. In applications
to BVPs representation $L = L_0 + K$ of a differential operator $L$ means that
($L_0$-compact) perturbation $K$ can change coefficients of low order terms of a
differential expression $L_0$ while boundary conditions remain unchanged. On the
other hand, under the conditions of Theorem~\ref{th:BarYak} an operator $L$ is a
singular (=non-additive) perturbation of $L_0$, in general, i.e. $\dom L \ne \dom
L_0$.

To describe the area of applicability of Theorems~\ref{Keldysh_th} and
\ref{th:BarYak} to BVPs let us consider the following simple example.
%
%
\begin{example}\label{Ex_S-L_oper_Intro}
Let $L_0$ be the Dirichlet realization of $-d^2/dx^2$ in $L^2[0,1]$, i.e.
\begin{equation}
	\dom L_0= \dom D^2_0=\{f\in W^{2,2}[0,1]:\ f(0)=f(1) = 0\},
\end{equation}
and $K: f\to q f$ where $q$ is complex valued, $q\in L^2[0,1]$. Then the Keldysh
theorem ensures completeness of $L=L_0+ K = D^2_0 + q$ in $L^2[0,1]$.

At the same time, one could not reach effect described in Theorem~\ref{th:BarYak}
by means of changing boundary conditions: each BVP for $-d^2/dx^2 + q$ with
non-degenerate BC is complete in $L^2[0,1]$ due to~\cite[Theorem 1.3.1]{Mar86}.
\end{example}
%
%
\noindent Similar effect for Dirac operator with $Q=0$ is discussed in
Example~\ref{ex:Dirac_with_Q=0}.

To treat these examples in general framework of BVPs we first recall definition
of a dual pair of operators and its proper extensions.
%
%
\begin{definition}
\begin{itemize}
\item[(i)] A pair $\{S_1,S_2\}$ of closed densely defined operators in $\fH$ is
called a dual pair of operators if $S_1 \subset S_2^* \ (\Longleftrightarrow
S_2\subset S_1^*)$.
\item[(ii)] An operator $T$ is called a proper extension of the dual pair $\{S_1,
S_2\}$ and is put in the class $\Ext\{S_1,S_2\}$ if $S_1 \subset T \subset S_2^*$.
\end{itemize}
\end{definition}
%
%
In connection with Theorem~\ref{th:BarYak} the following problem naturally
arises.

\textbf{Problem 1.} Given a dual pair of operators $\{S_1,S_2\}$ find all
peculiar pairs of proper extensions $T, \wt{S} \in \Ext \{S_1,S_2\}$ (i.e. such
operators that $T$ is normal and $\wt{S}$ is peculiarly complete) for which the
resolvent difference $(T - \l)^{-1} - (\wt{S} - \l)^{-1}$ is one-dimensional.

Note that in comparison with the assumptions of Theorem~\ref{th:BarYak} we restrict
the class of perturbations $\wt{S}$ by the class $\Ext \{S_1,S_2\}$ assuming that
it contains a normal extension $T$. Example 1 demonstrates significance of this
restriction. Namely, \textbf{Problem 1} has negative solution for a dual pair
$\{S, S\}$, where $S = D^2_{\min}$, $\dom D^2_{\min} = W^{2,2}_0[0,1]$, is the
minimal symmetric operator generated by the expression $-d^2/dx^2$. At the same
time, in accordance with Theorem~\ref{th:BarYak} proper selfadjoint extension $T =
D_0^2$ of $S$, where $D_0^2$ is the Dirichlet realization of $-d^2/dx^2$ in
$L^2[0,1]$, has rank one peculiar perturbation $\wt{S}$, which necessarily is not
a proper extension of $S$.

On the other hand, Theorem~\ref{th:one.dim.perturb.normal} shows that
\textbf{Problem 1} \emph{has an affirmative solution for the dual pair
$\{L_{\min}(B,0), L_{\min}(B^*,0)\}$}. Note in this connection that, in accordance
with Proposition~\ref{prop:normal=const}, a normal extension of a dual pair
$\{L_{\min}(B,Q), L_{\min}(B^*,Q)\}$ exists if and only if $Q=\const$.

The paper is organized as follows. In Section~\ref{sec:gen.prop} we find explicit
formula for the rank of the resolvent difference of arbitrary operators $L_{C,D}
(B,Q)$ and $L_{\wt{C},\wt{D}}(B,Q)$ in general $n \times n$ case. Namely, we show
that it is equal to $\rank \begin{pmatrix} C & D \\ \wt{C} & \wt{D}
\end{pmatrix} - n$. We also refine this formula in the case of $n=2$ and special
boundary conditions~\eqref{eq:wt.U1.U2.d1.h1.intro} for one of the operators.

In Section 3 we investigate Riesz basis property of operators $L_{V_1,V_2}
(B,Q)$ in the case of quasi-periodic boundary conditions \eqref{eq:U1.U2.d1.d2.intro}
and under certain assumptions on $Q$. In particular, we indicate conditions on
$Q$ ensuring similarity of such an operator either to a normal or to almost
normal operator.

In Section~\ref{sec:rank.one} we prove our main results on peculiar completeness
of one dimensional perturbations of operators $L_{V_1,V_2}(B,Q)$ with BC
\eqref{eq:U1.U2.d1.d2.intro}. In particular, we prove here Theorem
\ref{th:one.dim.perturb.normal}. Note that Theorem~\ref{th:one.dim.perturb.normal},
makes it reasonable a discussion of two other problems: \textbf{Problem 2}
and \textbf{Problem 3}, weaker versions of \textbf{Problem 1}. Namely, we
replace in formulation of \textbf{Problem 1} a normality of $T$ by one of its
weaker properties: similarity to a normal or almost normal operator, or just to
a property of $T$ to have the Riesz basis property with parentheses.

We show in Theorems~\ref{th:one.dim.perturb.similar_to_normal},
\ref{th:one.dim.perturb.riesz} that in opposite to \textbf{Problem 1}, both
\textbf{Problems 2 and 3} have an affirmative solution for a wide class of
potential matrices $Q$. Moreover, we discuss here {\bf Problem 1} for Dirac
operator with a non-trivial selfadjoint $2\times 2$ potential matrix $Q = Q^*$
and show that for a wide class of BVP the corresponding operator is complete
only simultaneously with its adjoint (see Example~\ref{ex:Dirac_with_Q=0}).

\textbf{Notation.} Let $T$ be a closed densely defined operator in a Hilbert space
$\fH$; $\sigma(T)$ and $\rho(T)=\bC \setminus \sigma(T)$ denote the spectrum and
resolvent set of the operator $T$, respectively; $\fS_p(\fH)$, $p\in [1,\infty]$,
denote the Neumann-Schatten ideal of the algebra $\cB(\fH)$ of bounded operators.
$\bD_r(z_0) := \{z \in \bC: |z - z_0| < r\}$ denotes the disc of the radius $r$
centered at $z_0$; $\bD_r := \bD_r(0)$.
%
%
\section{Resolvent difference properties of the operators $L_{C,D}(B,Q)$}
\label{sec:gen.prop}
%
\subsection{Formula for the rank of the resolvent difference}
%
%
In this subsection we consider operators $L_{C,D} := L_{C,D}(B,Q)$ associated with
BPV~\eqref{eq:system.nxn}--\eqref{eq:BC.nxn} in general $n \times n$ case. We will
find explicit formula for the rank of the resolvent difference of any two such
operators. Recall that for a bounded operator $A$ acting in a Hilbert space $\fH$
its rank is a dimension of its range, $\rank A := \dim (\ran A)$.

Let $\l \in \bC$ and $\Phi(\cdot, \l) \in AC\([0,1]; \bC^{n \times n}\)$ be a
fundamental matrix of the system~\eqref{eq:system.nxn}, i.e.
\begin{align} \label{eq:Phi.l.def}
 -iB^{-1} \Phi'(x, \l) + Q(x) \Phi(x, \l) = \l \Phi(x, \l),
 \ \ \text{for a.e.}\ x \in [0,1], \quad \Phi(0,\l) = I_n.
\end{align}
It is well-known that $\Phi(x, \l)$ is nonsingular for all $x \in [0,1]$ and
thus, $\Phi^{-1}(\cdot, \l) \in AC\([0,1]; \bC^{n \times n}\)$.

In what follows we denote by $R_{C,D}(\l) := (L_{C,D} - \l)^{-1}$ the resolvent
of the operator $L_{C,D}$ associated to the BVP~\eqref{eq:system.nxn}--\eqref{eq:BC.nxn}.
First we recall a simple lemma from~\cite{LunMal14}.
%
%
\begin{lemma} \label{lem:resolv.gen}
\cite[Corollary~4.2]{LunMal14}
Let $\l \in \rho\(L_{C,D}\)$. Then
\begin{equation} \label{eq:resolv.formula}
 \(R_{C,D}(\l)f\)(x)
 = (K_{\l} f)(x) - \Phi(x,\l) M_{C,D}(\l) (K_{\l}f)(1),
\end{equation}
where
\begin{align}
 \label{eq:MU.def}
 M_{C,D}(\l) &:= (C+D\Phi(1,\l))^{-1} D, \\
 \label{eq:Klf.def}
 (K_{\l}f)(x) &:= \Phi(x,\l) \int_0^x \Phi^{-1}(t,\l) i B f(t) dt.
\end{align}
\end{lemma}
%
%
Alongside the operator $L_{C,D}$ we consider the operator $L_{\wt{C},\wt{D}} :=
L_{\wt{C},\wt{D}}(B,Q)$ associated to equation~\eqref{eq:system.nxn} subject to
the boundary conditions
\begin{equation} \label{eq:wtBC.nxn}
 \wt{C} y(0) + \wt{D} y(1) = 0, \quad \wt{C}, \wt{D} \in \bC^{n \times n},
 \quad \rank(\wt{C} \ \ \wt{D}) = n.
\end{equation}
The following formula for the rank of the resolvent difference is immediately
implied by Lemma~\ref{lem:resolv.gen}.
%
%
\begin{lemma} \label{lem:rank.R-wtR}
Let $\l \in \rho(L_{C,D}) \cap \rho(L_{\wt{C},\wt{D}})$. Then
\begin{equation} \label{eq:rankR=rankM}
 \rank \bigl(R_{\wt{C},\wt{D}}(\l) - R_{C,D}(\l)\bigr) = \rank \widehat{M}(\l),
\end{equation}
where
\begin{equation} \label{eq:whM.def}
 \widehat{M}(\l) := M_{C,D}(\l) - M_{\wt{C},\wt{D}}(\l).
\end{equation}
Moreover, if common rank in~\eqref{eq:rankR=rankM} is equal to 1 then
$\widehat{M}(\l)$ admits representation
\begin{equation} \label{eq:whM=a.b}
 \widehat{M}(\l) = \alpha(\l) \cdot \beta(\l)^*
 = \(\alpha_j(\l) \ol{\beta_k(\l)}\)_{j,k=1}^n,
\end{equation}
for certain vector functions $\alpha, \beta : \bC \rightarrow \bC^n$, and
for any $f \in L^2\([0,1]; \bC^n\)$ we have
\begin{equation} \label{eq:R-wtR=a.b}
 \bigl(R_{\wt{C},\wt{D}}(\l) - R_{C,D}(\l)\bigr) f =
 \bigl(f, \Psi^*(\cdot, \l) \beta(\l) \bigr)_{L^2\([0,1]; \bC^n\)} \cdot
 \Phi(\cdot, \l) \alpha(\l),
\end{equation}
where
\begin{equation} \label{eq:Psi.def}
 \Psi(\cdot,\l) := i \Phi(1, \l) \Phi^{-1}(\cdot, \l) B.
\end{equation}
\end{lemma}
%
%
\begin{proof}
\textbf{(i)} It follows from Lemma~\ref{lem:resolv.gen}
(formula~\eqref{eq:resolv.formula}) that
\begin{equation}\label{eq:Rl-wt.Rl}
 \bigl(R_{\wt{C},\wt{D}}(\l)f - R_{C,D}(\l)f\bigr)(x) =
 \Phi(x, \l) \bigl(M_{C,D}(\l) - M_{\wt{C},\wt{D}}(\l)\bigr)
 \bigl[(K_{\l}f)(1)\bigr],
\end{equation}
for any $f \in \fH := L^2\([0,1]; \bC^n\)$. It easily follows from definition of
$K_{\l}$ (formula~\eqref{eq:Klf.def}) that
\begin{equation} \label{eq:ran.Kl}
 \{(K_{\l}f)(1) : f \in \fH \} = \bC^n.
\end{equation}
Namely, for $u \in \bC^n$, $(K_{\l}f)(1) = u$, if we set $f(x) = \Psi^{-1}(x, \l) u$.
Since $\Phi(\cdot, \l), \Phi^{-1}(\cdot, \l) \in AC\([0,1]; \bC^{n \times n}\)$,
formula~\eqref{eq:rankR=rankM} immediately follows from~\eqref{eq:Rl-wt.Rl},
\eqref{eq:ran.Kl} and~\eqref{eq:whM.def}.

\textbf{(ii)} If common rank in~\eqref{eq:rankR=rankM} is equal to 1 then
$\widehat{M}(\l)$ has rank 1 and thus admits representation~\eqref{eq:whM=a.b}.
It follows now from~\eqref{eq:Rl-wt.Rl} and definition of $K_{\l}$ and
$\Psi(\cdot, \l)$ (formulas~\eqref{eq:Klf.def} and~\eqref{eq:Psi.def})
that for any $f \in \fH$
\begin{align*}
 \bigl(R_{\wt{C},\wt{D}}(\l) - R_{C,D}(\l)\bigr) f
 &= \Phi(\cdot, \l) \alpha(\l) \cdot \beta(\l)^*
 \cdot \int_0^1 \Psi(t,\l) f(t) d \\
 &= \Phi(\cdot, \l) \alpha(\l) \cdot \int_0^1
 \langle f(t), \Psi^*(t, \l) \beta(\l) \rangle_{\bC^n} dt \\
 &= \bigl( f, \Psi^*(\cdot, \l) \beta(\l) \bigr)_{\fH} \cdot
 \Phi(\cdot, \l) \alpha(\l),
\end{align*}
which finishes the proof.
\end{proof}
The following result gives explicit formula for the rank of the resolvent
difference of operators $L_{\wt{C},\wt{D}}$ and $L_{C,D}$ in terms of marices
$C,D,\wt{C},\wt{D}$.
%
%
\begin{proposition} \label{prop:rankR=rankCD}
Let $\l \in \rho(L_{C,D}) \cap \rho(L_{\wt{C},\wt{D}})$. Then
\begin{equation} \label{eq:rankR=rankCD}
 \rank \(R_{\wt{C},\wt{D}}(\l) - R_{C,D}(\l)\) =
 \rank\begin{pmatrix}C & D \\ \wt{C} & \wt{D}\end{pmatrix} - n.
\end{equation}
\end{proposition}
%
%
\begin{proof}
Let us set $A := A(\l) := C + \Phi(1, \l) D$ and $\wt{A} := \wt{A}(\l) := \wt{C}
+ \Phi(1, \l) \wt{D}$. Note that matrices $A$ and $\wt{A}$ are nonsingular
since $\l \in \rho(L_{C,D}) \cap \rho(L_{\wt{C},\wt{D}})$. Taking this into
account we get
\begin{align} \label{eq:rankCD=rankM}
 \rank \begin{pmatrix} C & D \\ \wt{C} & \wt{D}\end{pmatrix}
 &= \rank\begin{pmatrix} C + \Phi(1,\l) D & D \notag \\
 \wt{C} + \Phi(1,\l) \wt{D} & \wt{D} \end{pmatrix}
 = \rank \begin{pmatrix} A & D \\ \wt{A} & \wt{D}\end{pmatrix} \\
 &= \rank \(\begin{pmatrix} A & 0 \\ 0 & \wt{A}\end{pmatrix}
 \begin{pmatrix} I_n & A^{-1} D \\ I_n & \wt{A}^{-1} \wt{D}\end{pmatrix}\)
 = \rank \begin{pmatrix} I_n & A^{-1} D \\
 I_n & \wt{A}^{-1} \wt{D}\end{pmatrix} \notag \\
 &= \rank \begin{pmatrix} 0 & A^{-1} D - \wt{A}^{-1} \wt{D} \\
 I_n & \wt{A}^{-1} \wt{D}\end{pmatrix}
 = n + \rank\(A^{-1} D - \wt{A}^{-1} \wt{D}\) \notag \\
 &= n + \rank\(\bigl(C + \Phi(1, \l) D\bigr)^{-1} D
 - \bigl(\wt{C} + \Phi(1, \l) \wt{D}\bigr)^{-1} \wt{D}\) \notag \\
 &= n + \rank\(M_{C,D}(\l) - M_{\wt{C},\wt{D}}(\l)\)
 = n + \rank \widehat{M}(\l).
\end{align}
Formula~\eqref{eq:rankR=rankCD} now follows from~\eqref{eq:rankR=rankM}
and~\eqref{eq:rankCD=rankM}.
\end{proof}
%
%
\subsection{Resolvent difference properties for $2 \times 2$ system}
%
%
Let $\Phi(x,\l)$ be the fundamental matrix of the system~\eqref{eq:system} defined
in the previous subsection and
\begin{equation} \label{eq:Phi.def.2x2}
 \Phi(x,\l) :=
 \begin{pmatrix}
 \Phi_1(x,\l) &
 \Phi_2(x,\l)
 \end{pmatrix}, \quad
 \Phi_j(x,\l) :=
 \begin{pmatrix}
 \varphi_{1j}(x,\l) \\
 \varphi_{2j}(x,\l)
 \end{pmatrix}, \quad j\in\{1,2\}.
\end{equation}
The eigenvalues of the problem~\eqref{eq:system}--\eqref{eq:BC} are
the roots of the characteristic equation $\Delta(\l) := \det
U(\l)=0$, where
\begin{equation}\label{eq:U}
 U(\l) :=
 \begin{pmatrix}
 U_1(\Phi_1(x,\l)) & U_1(\Phi_2(x,\l)) \\
 U_2(\Phi_1(x,\l)) & U_2(\Phi_2(x,\l))
 \end{pmatrix} =:
 \begin{pmatrix}
 u_{11}(\l) & u_{12}(\l) \\
 u_{21}(\l) & u_{22}(\l)
 \end{pmatrix}.
\end{equation}
Further, let us set
\begin{equation}\label{eq:Ajk.Jjk}
 A_{jk} := \begin{pmatrix} a_{1j} & a_{1k} \\ a_{2j} & a_{2k} \end{pmatrix}
 \quad\text{and}\quad
 J_{jk} := \det A_{jk}, \quad j,k\in\{1,\ldots,4\}.
\end{equation}
Note, that boundary conditions~\eqref{eq:BC} takes the form~\eqref{eq:BC.nxn}
if we set $C := A_{12}$ and $D := A_{34}$. In particular, $U(\l) = C + D \Phi(1, \l)$.

Taking into account notations~\eqref{eq:Ajk.Jjk} we arrive at the following expression
for the characteristic determinant:
\begin{equation}\label{eq:Delta}
 \Delta(\l) = J_{12} + J_{34}e^{i(b_1+b_2)\l}
 + J_{32}\varphi_{11}(\l) + J_{13}\varphi_{12}(\l)
 + J_{42}\varphi_{21}(\l) + J_{14}\varphi_{22}(\l),
\end{equation}
where $\varphi_{jk}(\l) := \varphi_{jk}(1,\l)$. If $Q=0$ then
$\varphi_{12}(x,\l) = \varphi_{21}(x,\l) = 0$, and the
characteristic determinant $\Delta_0(\cdot)$ has the form
\begin{equation}\label{eq:Delta0}
 \Delta_0(\l) = J_{12} + J_{34}e^{i(b_1+b_2)\l}
 + J_{32}e^{ib_1\l} + J_{14}e^{ib_2\l}.
\end{equation}

In what follows we denote by $R_{U_1,U_2}(\l) := (L_{U_1,U_2} - \l)^{-1}$ the
resolvent of the operator $L_{U_1,U_2}$ associated to the
BVP~\eqref{eq:system}--\eqref{eq:BC}.
Straightforward calculations lead to explicit formula for the matrix function
$M_{U_1,U_2}(\l) := M_{C,D}(\l)$ given by~\eqref{eq:MU.def}, via determinants
$J_{jk}$ from~\eqref{eq:Ajk.Jjk}.
%
%
\begin{lemma} \label{lem:MU}
Let $\l\in \rho(L_{U_1,U_2})$. Then $\Delta(\l) \ne 0$,
$M_{U_1,U_2}(\l)$ is well defined and admits the following
representation
\begin{equation}\label{eq:MU}
 M_{U_1,U_2}(\l) = \frac{1}{\Delta(\l)} \begin{pmatrix}
 J_{32} + J_{34} \varphi_{22}(\l) & J_{42} - J_{34} \varphi_{12}(\l) \\
 J_{13} - J_{34} \varphi_{21}(\l) & J_{14} + J_{34} \varphi_{11}(\l) \\
 \end{pmatrix}.
\end{equation}
Moreover,
\begin{equation}\label{eq:detMU}
 \det M_{U_1,U_2}(\l) = \frac{\det D}{\det(C+D\Phi(1,\l))}
 = \frac{J_{34}}{\Delta(\l)},
\end{equation}
where $C = A_{12}$ and $D = A_{34}$.
\end{lemma}
%
%
\begin{proof}
According to definition~\eqref{eq:U} $\Delta(\l) := \det U(\l) =
\det(C+D \Phi(1,\l))$ and
\begin{equation} \label{eq:C+D.Phi=2x2}
 C+D \Phi(1,\l)
 = \begin{pmatrix}
 a_{11} + a_{13} \varphi_{11}(\l) + a_{14} \varphi_{21}(\l) &
 a_{12} + a_{13} \varphi_{12}(\l) + a_{14} \varphi_{22}(\l) \\
 a_{21} + a_{23} \varphi_{11}(\l) + a_{24} \varphi_{21}(\l) &
 a_{22} + a_{23} \varphi_{12}(\l) + a_{24} \varphi_{22}(\l) \\
 \end{pmatrix}.
\end{equation}
Hence the following formula for the inverse matrix holds
\begin{multline} \label{eq:inv.C+D.Phi=2x2}
 (C+D \Phi(1,\l))^{-1} = \\
 \frac{1}{\Delta(\l)} \begin{pmatrix}
 a_{22} + a_{23} \varphi_{12}(\l) + a_{24} \varphi_{22}(\l) &
 -(a_{12} + a_{13} \varphi_{12}(\l) + a_{14} \varphi_{22}(\l)) \\
 -(a_{21} + a_{23} \varphi_{11}(\l) + a_{24} \varphi_{21}(\l)) &
 a_{11} + a_{13} \varphi_{11} (\l) + a_{14} \varphi_{21}(\l) \\
 \end{pmatrix}.
\end{multline}
Multiplying~\eqref{eq:inv.C+D.Phi=2x2} by $D$ from the left we
arrive at formula~\eqref{eq:MU}. E.g. for the first entry we have
\begin{align}
 \Delta(\l) \bigl[M_{U_1,U_2}(\l)\bigr]_{11}
 &= (a_{22} + a_{23} \varphi_{12}(\l) + a_{24} \varphi_{22}(\l)) a_{13}
 - (a_{12} + a_{13} \varphi_{12}(\l) + a_{14} \varphi_{22}(\l)) a_{23} \notag \\
 &= (a_{22} a_{13} - a_{12} a_{23}) + (a_{23} a_{13} - a_{13} a_{23}) \varphi_{12}(\l)
 + (a_{24} a_{13} - a_{14} a_{23}) \varphi_{22}(\l) \notag \\
 &= J_{32} + J_{34} \varphi_{22}(\l).
\label{eq:MU.11}
\end{align}
The rest equalities in~\eqref{eq:MU} are verified similarly.
\end{proof}
Alongside the operator $L_{U_1,U_2}$ we consider the operator
$L_{\wt{U}_1,\wt{U}_2} := L_{\wt{U}_1,\wt{U}_2}(B,Q)$ associated to
equation~\eqref{eq:system} subject to the boundary conditions
\begin{equation} \label{eq:wtBC}
 \wt{U}_j(y) := \wt{a}_{j 1}y_1(0) + \wt{a}_{j 2}y_2(0)
 + \wt{a}_{j 3}y_1(1) + \wt{a}_{j 4}y_2(1)= 0,
 \quad j \in \{1,2\}.
\end{equation}
Similarly to~\eqref{eq:Delta} we have the following formula for the
characteristic determinant $\wt{\Delta}(\cdot)$ of the operator
$L_{\wt{U}_1,\wt{U}_2}$,
\begin{equation}\label{eq:wtDelta}
 \wt{\Delta}(\l) = \wt{J}_{12} + \wt{J}_{34}e^{i(b_1+b_2)\l}
 + \wt{J}_{32}\varphi_{11}(\l) + \wt{J}_{13}\varphi_{12}(\l)
 + \wt{J}_{42}\varphi_{21}(\l) + \wt{J}_{14}\varphi_{22}(\l),
\end{equation}
where
\begin{equation} \label{eq:wtJjk.def}
 \wt{J}_{jk} := \det \wt{A}_{jk}, \quad \wt{A}_{jk} := \begin{pmatrix}
 \wt{a}_{1j} & \wt{a}_{1k} \\
 \wt{a}_{2j} & \wt{a}_{2k}
 \end{pmatrix}, \quad j,k \in \{1,\ldots,4\}.
\end{equation}
Note that $L_{\wt{U}_1, \wt{U}_2} = L_{\wt{C}, \wt{D}}$ with $\wt{C} :=
\wt{A}_{12}$ and $\wt{D} := \wt{A}_{34}$. The following result immediately follows
from Proposition~\ref{prop:rankR=rankCD}.
%
%
\begin{corollary} \label{cor:one.dim.crit.2x2}
Let $L_{U_1, U_2} \ne L_{\wt{U}_1,\wt{U}_2}$ and $\l \in \rho(L_{U_1, U_2}) \cap
\rho(L_{\wt{U}_1,\wt{U}_2})$. Then the resolvent difference
$R_{\wt{U}_1,\wt{U}_2}(\l) - R_{U_1,U_2}(\l)$ is one-dimensional if and only if
\begin{equation} \label{eq:det.A.wtA=0}
 \det \begin{pmatrix} A_{12} & A_{34} \\ \wt{A}_{12} & \wt{A}_{34}
 \end{pmatrix} = 0,
\end{equation}
which in turn is equivalent to
\begin{equation} \label{eq:Jjk.wtJjk=0}
 J_{12} \wt{J}_{34} + \wt{J}_{12} J_{34} +
 J_{13} \wt{J}_{42} + \wt{J}_{13} J_{42} +
 J_{14} \wt{J}_{23} + \wt{J}_{14} J_{23} = 0.
\end{equation}
\end{corollary}
%
%
\begin{proof}
Since $\rank(A_{12} \ \ A_{34}) = \rank(\wt{A}_{12} \ \ \wt{A}_{34}) = 2$ and
$L_{U_1,U_2} \ne L_{\wt{U}_1, \wt{U}_2}$ it follows that
\begin{equation}
 r:= \rank\begin{pmatrix} A_{12} & A_{34} \\ \wt{A}_{12} & \wt{A}_{34}
 \end{pmatrix} \in \{3, 4\}.
\end{equation}
Hence, $r = 3$ if and only if condition~\eqref{eq:det.A.wtA=0} holds.
In turn, $r = 3$ is equivalent to the fact that the resolvent difference
$R_{\wt{U}_1,\wt{U}_2}(\l) - R_{U_1,U_2}(\l)$ is one-dimensional due to
Proposition~\ref{prop:rankR=rankCD}.

Finally, applying Laplace expansion by the first 2 rows to the determinant
in~\eqref{eq:det.A.wtA=0} and taking into account definition of $J_{jk}$
and $\wt{J}_{jk}$ we get equivalence of~\eqref{eq:det.A.wtA=0}
and~\eqref{eq:Jjk.wtJjk=0}.
\end{proof}
%
%
\subsection{Special boundary conditions}
%
%
Next we consider system~\eqref{eq:system}
\begin{equation}\label{eq:system2}
 \cL y = -i B^{-1} y'+Q(x)y=\l y, \qquad y=\col(y_1,y_2), \qquad x\in[0,1].
\end{equation}
subject to the special boundary conditions
\begin{equation}\label{eq:U1.U2.h0.h1}
 \wt{U}_1(y):=y_1(0)-h_1 y_2(0)=0,\qquad
 \wt{U}_2(y)= y_1(1)-h_2 y_2(0)=0.
\end{equation}
Here $Q$ is given by~\eqref{eq:BQ} and $h_1,h_2\in\bC\setminus\{0\}$.

Denote by $L_{\wt{U}_1,\wt{U}_2} = L_{\wt{U}_1,\wt{U}_2}(B, Q)$ the operator
associated to the problem~\eqref{eq:system2}--\eqref{eq:U1.U2.h0.h1} in $\fH
= L^2([0,1]; \bC^2)$.

In the following proposition we indicate simple algebraic condition on coefficients
of general problem~\eqref{eq:system}--\eqref{eq:BC} ensuring that the resolvent
difference of operators $L_{U_1,U_2}$ and $L_{\wt{U}_1,\wt{U}_2}$ is one-dimensional.
Moreover, we give explicit form of this resolvent difference.
%
%
\begin{proposition}\label{prop:resolv.dif}
Let $L_{\wt{U}_1,\wt{U}_2}\ne L_{U_1,U_2}$ and $\l \in \rho(L_{U_1,U_2}) \cap
\rho(L_{\wt U_1,\wt U_2})$.

\begin{itemize}
\item[(i)] Then the resolvent difference $R_{\wt{U}_1,\wt{U}_2}(\l) -
R_{U_1,U_2}(\l)$ is one-dimensional if and only if
\begin{equation} \label{eq:h1.h0.one.dim}
 J_{34} h_2 + J_{14} h_1 = J_{42}.
\end{equation}
\item[(ii)] Let condition~\eqref{eq:h1.h0.one.dim} is fulfilled and in addition
\begin{equation} \label{eq:gamma.l.def}
 \gamma(\l) := J_{14}+J_{34} \varphi_{11}(\l) \ne 0,
\end{equation}
then the resolvent difference $R_{\wt{U}_1,\wt{U}_2}(\l) - R_{U_1,U_2}(\l)$
admits representation~\eqref{eq:R-wtR=a.b} with the vector-functions $\alpha =:
\col(\alpha_1, \alpha_2)$ and $\beta =: \col(\beta_1, \beta_2)$ given by
\begin{align}
\label{eq:alpha12}
 \alpha_1(\l) &= h_1 - \frac{J_{34} \wt{\Delta}(\l)}{\gamma(\l)},
 \qquad \alpha_2(\l) = 1, \\
\label{eq:beta12}
 \ol{\beta_1(\l)} &= \frac{J_{13}-J_{34} \varphi_{21}(\l)}{\Delta (\l)}
 - \frac{1}{\wt{\Delta}(\l)},
 \qquad \ol{\beta_2(\l)} = \frac{\gamma(\l)}{\Delta(\l)}.
\end{align}
\end{itemize}
\end{proposition}
%
%
\begin{proof}
\textbf{(i)} It follows from~\eqref{eq:U1.U2.h0.h1} that $\wt{A}_{12} =
\begin{pmatrix} 1 & -h_1 \\ 0 & -h_2 \end{pmatrix}$ and $\wt{A}_{34} =
\begin{pmatrix} 0 & 0 \\ 1 & 0 \end{pmatrix}$. Hence by definition of $\wt{J}_{jk}$
we have
\begin{equation} \label{eq:J12..J34}
 \wt{J}_{12} = -h_2, \ \ \wt{J}_{13} = 1, \ \ \wt{J}_{32} = h_1, \ \
 \wt{J}_{14} = \wt{J}_{42} = \wt{J}_{34} = 0.
\end{equation}
Thus condition~\eqref{eq:Jjk.wtJjk=0} transforms into~\eqref{eq:h1.h0.one.dim}.
Corollary~\ref{cor:one.dim.crit.2x2} now finishes the proof of part (i).

\textbf{(ii)} Due to (i) and Lemma~\ref{lem:rank.R-wtR}
condition~\eqref{eq:h1.h0.one.dim} yields that $\rank \widehat{M}(\l) = 1$.
Hence $\widehat{M}(\l)$ admits representation~\eqref{eq:whM=a.b} which for $n=2$
turns into
\begin{equation} \label{eq:whM=a.b.n=2}
 \widehat{M}({\l}) = \begin{pmatrix}
 \alpha_1(\l) \ol{\beta_1(\l)} & \alpha_1(\l) \ol{\beta_2(\l)} \\
 \alpha_2(\l) \ol{\beta_1(\l)} & \alpha_2(\l) \ol{\beta_2(\l)} \\
 \end{pmatrix}.
\end{equation}
Let us verify formulas~\eqref{eq:alpha12}--\eqref{eq:beta12} for
$\alpha_1(\l)$, $\alpha_2(\l)$, $\beta_1(\l)$, $\beta_2(\l)$. It follows
from~\eqref{eq:Delta},~\eqref{eq:MU} and~\eqref{eq:J12..J34} that
\begin{equation} \label{eq:Delta.h0.h1}
 \wt{\Delta}(\l) = -h_2 + h_1 \varphi_{11}(\l) + \varphi_{12}(\l),
 \quad M_{\wt{U}_1,\wt{U}_2}(\l) = \frac{1}{\wt{\Delta}(\l)}
 \begin{pmatrix} h_1 & 0 \\ 1 & 0 \end{pmatrix}.
\end{equation}
Put
\begin{equation} \label{eq:wtM.Delta.phi}
 M_{U_1,U_2}(\l) =:
 \begin{pmatrix} m_{11} & m_{12} \\ m_{21} & m_{22} \end{pmatrix},
\end{equation}
where for convenience we omitted dependency on $\l$. It follows
from~\eqref{eq:Delta.h0.h1} and~\eqref{eq:wtM.Delta.phi} that
\begin{equation} \label{eq:M-wtM.spec}
 \widehat{M}({\l}) =
 \begin{pmatrix}
 m_{11} - \frac{h_1}{\wt{\Delta}(\l)}& m_{12} \\
 m_{21} - \frac{1}{\wt{\Delta}(\l)} & m_{22}
 \end{pmatrix} =: \begin{pmatrix}
 \widehat{m}_{11} & \widehat{m}_{12} \\
 \widehat{m}_{21} & \widehat{m}_{22}
 \end{pmatrix}.
\end{equation}
If $\widehat{m}_{22} \ne 0$ and $\det \widehat{M}({\l}) = 0$ it can
easily be seen that representation~\eqref{eq:whM=a.b.n=2} takes place
for instance with
\begin{equation} \label{eq:a1.a2.b1.b2.sjk}
 \alpha_1(\l) = \frac{\widehat{m}_{12}}{\widehat{m}_{22}},
 \quad \alpha_2(\l) = 1, \quad \ol{\beta_1(\l)} = \widehat{m}_{21},
 \quad \ol{\beta_2(\l)} = \widehat{m}_{22}.
\end{equation}
It follows from~\eqref{eq:M-wtM.spec},~\eqref{eq:MU} and
definition~\eqref{eq:gamma.l.def}
of $\gamma(\l)$ that
\begin{equation} \label{eq:s12.s21.s22}
 \widehat{m}_{12} = \frac{J_{42} - J_{34} \varphi_{12}(\l)}{\Delta(\l)}, \quad
 \widehat{m}_{21} = \frac{J_{13} - J_{34} \varphi_{21}(\l)} {\Delta(\l)}
 - \frac{1}{\wt{\Delta}(\l)}, \quad
 \widehat{m}_{22} = \frac{\gamma(\l)}{\Delta(\l)}.
\end{equation}
Since $\gamma(\l) \ne 0$ it follows that $\widehat{m}_{22} \ne 0$.
Formula~\eqref{eq:beta12} now immediately follows
from~\eqref{eq:a1.a2.b1.b2.sjk} and~\eqref{eq:s12.s21.s22}. For
$\alpha_1(\l)$ we derive
from~\eqref{eq:s12.s21.s22},~\eqref{eq:h1.h0.one.dim},~\eqref{eq:Delta.h0.h1}
and~\eqref{eq:gamma.l.def}
\begin{align} \label{eq:alpha1}
 \alpha_1(\l) &= \frac{\widehat{m}_{12}}{\widehat{m}_{22}}
 = \frac{J_{42} - J_{34} \varphi_{12}(\l)}{\gamma(\l)}
 = \frac{J_{14} h_1 + J_{34} (h_2 - \varphi_{12}(\l))}{\gamma(\l)} \notag \\
 &= \frac{J_{14} h_1 + J_{34} (h_1 \varphi_{11}(\l) - \wt{\Delta}(\l))}{\gamma(\l)}
 = h_1 - \frac{J_{34} \wt{\Delta}(\l)}{\gamma(\l)}.
\end{align}
This completes the proof.
\end{proof}
%
%
Next we show that for almost each BVP~\eqref{eq:system}--\eqref{eq:BC},
there exist BVP~\eqref{eq:system2}--\eqref{eq:U1.U2.h0.h1} such that the
corresponding resolvent difference is one-dimensional.
%
%
\begin{corollary} \label{cor:one.dim.regular}
Let $L_{U_1,U_2}$ be an operator associated to the
problem~\eqref{eq:system}--\eqref{eq:BC} and let $J_{jk}$ be defined
by~\eqref{eq:Ajk.Jjk}. Assume that among numbers $\{J_{14}, J_{42}, J_{34}\}$
either at least two are non-zero or all zero. Then there exists a pair
$\{h_1, h_2\} $ with $h_1 h_2\ne 0$ such that the resolvent difference
$R_{\wt{U}_1,\wt{U}_2}(\l) - R_{U_1,U_2}(\l)$ is one-dimensional.

In particular, the latter holds for any regular boundary conditions
$U_1, U_2$ (see Definition~\ref{def1.1}).
\end{corollary}
%
%
\begin{proof}
By Proposition~\ref{prop:resolv.dif} it suffices to choose $h_1, h_2 \ne 0$
satisfying~\eqref{eq:h1.h0.one.dim}. If all $J_{14}$, $J_{42}$, $J_{34}$ are
zero any pair $\{h_1, h_2\}$ with $h_1 h_2\ne 0$ is suitable. If at least two of
these numbers are non-zero, then existence of required numbers $\{h_1, h_2\}$ is
immediate from~\eqref{eq:h1.h0.one.dim}.

Now assume boundary conditions to be regular. In both cases $b_1/b_2 \in \bR$
and $b_1/b_2 \notin \bR$ it implies that $J_{14} J_{32} \ne 0$ (see Appendix).
In this case $J_{34}$ and $J_{42}$ cannot equal zero simultaneously, since
otherwise $J_{32}$ would be zero. And thus, among numbers $J_{14}$, $J_{42}$,
$J_{34}$ at least two are non-zero.
\end{proof}
%
%
\section{Riesz basis property for $2 \times 2$ system}
\label{sec:riesz}
%
%
Here we consider system~\eqref{eq:system},
\begin{equation} \label{eq:riesz.system}
 \cL y := -i B^{-1} y' + Q(x)y = \l y, \qquad y = \col(y_1,y_2), \qquad
 x \in [0,1],
\end{equation}
where matrices $B$ and $Q(\cdot)$ are given by
\begin{equation} \label{eq:riesz.BQ}
 B = \text{diag}(b_1, b_2), \quad b_1 b_2^{-1} \not \in \bR, \qquad
 Q = \begin{pmatrix} 0 & Q_{12} \\ Q_{21} & 0 \end{pmatrix} \in
 A(\bD_R; \bC^{2 \times 2}).
\end{equation}
Here $Q \in A(\bD_R; \bC^{2 \times 2})$ means that $Q_{12}$ and $Q_{21}$ admit
an analytic continuation to the disk $\bD_R$ for some sufficiently large $R$.

In this section we study Riesz basis property for the system of root vectors of
the operator $L_{U_1,U_2}(B,Q)$ generated by equation
\eqref{eq:riesz.system}--\eqref{eq:riesz.BQ} subject to the boundary conditions
\begin{equation} \label{eq:riesz.BC}
	U_j(y) := a_{j 1}y_1(0) + a_{j 2}y_2(0) + a_{j 3}y_1(1) + a_{j 4}y_2(1) = 0,
	\quad j \in \{1,2\}.
\end{equation}

First, we recall a special case of Theorem 3.2 from~\cite{Mal99} on existence
of a triangular transformation operator for a general system~\eqref{eq:system.nxn}
with analytical potential matrix $Q(\cdot)$. We set
\begin{equation} \label{eq:Q.norm}
 \|Q\| := \|Q\|_{C[0,1]} := \max \{\|Q_{12}\|_{C[0,1]}, \|Q_{21}\|_{C[0,1]}\}.
\end{equation}
%
%
\begin{proposition}~\cite{Mal99} \label{prop:trans.oper}
Assume that $e_{\pm}(\cdot,\l)$ are the solutions of the Cauchy problem for
system~\eqref{eq:riesz.system}--\eqref{eq:riesz.BQ} satisfying the initial
conditions $e_{\pm}(0,\l) = \binom{1}{\pm 1}$. Then they admit the following
triangular representations
\begin{equation} \label{eq:epm=Ke}
 e_{\pm}(x,\l) 
 = e^0_{\pm}(x,\l) + \int^x_0 K^{\pm}(x,t) e^0_{\pm}(t,\l)dt,
\end{equation}
where
\begin{equation}
 e^0_{\pm}(x,\l) := \binom{e^{i b_1 \l x}}{\pm e^{i b_2 \l x}}, \qquad
 K^{\pm} = \bigl(K^{\pm}_{jk}\bigr)_{j,k=1}^2 \in
 C^{\infty}(\Omega; \bC^{2 \times 2}),
\end{equation}
and $\Omega:=\{(x,t):0\le t\le x\le 1\}$. Moreover, the following estimates hold
\begin{equation}\label{eq:est-mate_for_K_jk}
 \| K_{jk}^{\pm} \|_{C(\Omega)} \le C_0 \|Q\| \cdot \exp(C_1 \|Q\|),
 \quad j,k \in \{1,2\},
\end{equation}
with some constants $C_0, C_1 > 0$.
\end{proposition}
%
%
Note that estimates~\eqref{eq:est-mate_for_K_jk} are easily extracted from the
proof of~\cite[Theorem~3.2]{Mal99}.

Let further
\begin{equation} \label{eq:riesz.Phi.def}
 \Phi(\cdot, \l) = \begin{pmatrix}
 \varphi_{11}(\cdot, \l) & \varphi_{12}(\cdot, \l) \\
 \varphi_{21}(\cdot,\l) & \varphi_{22}(\cdot,\l)
 \end{pmatrix} =: \begin{pmatrix}
 \Phi_1(\cdot, \l) & \Phi_2(\cdot, \l)
 \end{pmatrix}, \qquad \Phi(0, \l) = I_2,
\end{equation}
be a fundamental matrix solution of system~\eqref{eq:riesz.system}. Here
$\Phi_j(\cdot, \l)$, $j\in \{1,2\}$, is the $j$th column of $\Phi(\cdot, \l)$.

In the sequel we follow the scheme proposed in~\cite{LunMal16JMAA} for
investigating the Riesz basis property of BVP for Dirac type system $(B = B^*)$
with a summable potential matrix $Q$. The following result is similar to that
of Proposition 3.1 from~\cite{LunMal16JMAA}.
%
%
\begin{lemma}\label{prop:phi.jk=e+int}
Let $Q \in A(\bD_R; \bC^{2 \times 2})$. Then the functions $\varphi_{jk}(\cdot,
\l)$, $j,k \in \{1,2\}$, admit the following representations
\begin{align} \label{eq:phijk}
 \varphi_{jk}(x,\l) = \delta_{jk} e^{i b_k \l x}
 + \int_0^x R_{jk1}(x,t) e^{i b_1 \l t}dt
 + \int_0^x R_{jk2}(x,t) e^{i b_2 \l t}dt,
\end{align}
where $R_{jkh} \in C^{\infty}(\Omega)$ and there exists constants $C_0, C_1 > 0$
such that
\begin{equation} \label{eq:Rjkh}
 \| R_{jkh} \|_{C(\Omega)} \le C_0 \|Q\| \cdot \exp(C_1 \|Q\|),
 \quad j,k,h \in \{1,2\}.
\end{equation}
\end{lemma}
%
%
\begin{proof}
Comparing initial conditions and applying the Cauchy uniqueness theorem one
easily gets $2 \Phi_1(\cdot,\l) = 2 \begin{pmatrix} \varphi_{11}(\cdot, \l) \\
\varphi_{21}(\cdot,\l) \end{pmatrix} = e_{+}(\cdot, \l) + e_{-}(\cdot, \l)$.
Inserting in place of $e_{+}(\cdot, \l)$ and $e_{-}(\cdot, \l)$ their expressions
from~\eqref{eq:epm=Ke} one arrives at~\eqref{eq:phijk}--\eqref{eq:Rjkh} for
$k=1$. Relations~\eqref{eq:phijk}--\eqref{eq:Rjkh} for $k=2$ are proved similarly.
\end{proof}
%
%
Next we obtain a formula for the characteristic determinant $\Delta(\cdot)$ of
the BVP~\eqref{eq:riesz.system}--\eqref{eq:riesz.BC} similar to that used
in~\cite[Lemma 4.1]{LunMal16JMAA}.
%
%
\begin{lemma} \label{lem:Delta=Delta0+}
Let $Q \in A(\bD_R; \bC^{2 \times 2})$. The characteristic determinant
$\Delta(\cdot)$ of the BVP~\eqref{eq:riesz.system}--\eqref{eq:riesz.BC} is an
entire function admitting the following representation
\begin{equation}\label{eq:Delta=Delta0+}
 \Delta(\l) = \Delta_0(\l)
 + \int^1_0 g_1(t) e^{i b_1 \l t} dt
 + \int^1_0 g_2(t) e^{i b_2 \l t} dt.
\end{equation}
Here $g_{j} \in C^{\infty}[0,1]$, \ $j \in \{1,2\}$, and
\begin{equation} \label{eq:gj.in.C}
 \| g_{j} \|_{C[0,1]} \le C_0 \|Q\| \cdot \exp(C_1 \|Q\|), \qquad j \in \{1,2\},
\end{equation}
with some constants $C_0> 0$ and $C_1> 0$.
\end{lemma}
%
%
\begin{proof}
Inserting formulas~\eqref{eq:phijk} with $x=1$ into~\eqref{eq:Delta} and taking
expression~\eqref{eq:Delta0} for $\Delta_0(\cdot)$ into account, we arrive at
the expression~\eqref{eq:Delta=Delta0+} for $\Delta(\cdot)$.

Since $g_j(\cdot)$, $j \in \{1,2\}$, are linear combinations of the functions
$R_{jkh}(1,\cdot)$, $j,k,h \in \{1,2\}$, estimates~\eqref{eq:gj.in.C} easily
follow from estimates~\eqref{eq:Rjkh}.
\end{proof}
In the remaining part of the section we will study operator $L_{U_1,U_2}(B,Q)$
subject to the following "quasi-periodic" boundary conditions
\begin{equation} \label{eq:riesz.U1.U2.d1.d2}
 U_1(y) = y_1(0) - d_1 y_1(1) = 0, \quad
 U_2(y) = y_2(0) - d_2 y_2(1) = 0, \quad d_1, d_2 \in \bC \setminus \{0\}.
\end{equation}
First we study characteristic determinant $\Delta_0(\l)$ of the operator
$L_{U_1,U_2}(B,Q)$. Let us recall the following definition.
%
%
\begin{definition} \label{def:sequences}
\textbf{(i)} A sequence $\Lambda := \{\l_n\}_{n \in \bZ}$ of complex numbers is
said to be \textbf{separated} if for some positive $\delta > 0$,
\begin{equation} \label{eq:separ_cond}
 |\l_j - \l_k| > 2 \delta \quad \text{whenever}\quad j \ne k.
\end{equation}
In particular, all entries of a separated sequence are distinct.

\textbf{(ii)} The sequence $\Lambda$ is said to be \textbf{asymptotically
separated} if for some $n_0 \in \bN$ the subsequence $\Lambda_{n_0} :=
\{\l_n\}_{|n| > n_0}$ is separated.
\end{definition}
%
%
\begin{lemma} \label{lem:delta0}
Let $\Delta_0(\cdot)$ be the characteristic determinant of the problem
\eqref{eq:riesz.system},~\eqref{eq:riesz.U1.U2.d1.d2} with $Q = 0$, and let $b_1
b_2^{-1} \notin \bR$. Let also
\begin{equation} \label{eq:dj=e.gammaj}
 d_j = e^{-2 \pi i \gamma_j}, \quad \gamma_j = \alpha_j + i \beta_j,
 \quad \alpha_j \in [0,1), \quad \beta_j \in \bR, \quad j \in \{1,2\}.
\end{equation}
Then the following statements hold:
\begin{itemize}
\item[(i)] The sequence of zeros $\Lambda_0$ of $\Delta_0(\cdot)$ counting
multiplicity is of the form
\begin{equation} \label{eq:Lambda0}
 \Lambda_0 = \{\l_{nj}^0\}_{n \in \bZ, \, j \in \{1,2\}}, \quad
 \l_{nj}^0 = 2 \pi b_j^{-1} (\gamma_j + n), \quad n \in \bZ, \ j \in \{1,2\},
\end{equation}
In particular, the sequence $\Lambda_0$ is asymptotically separated and all
its entries but possibly one are simple.
\item[(ii)] Let $b_1 b_2^{-1} = c_1 + i c_2$, $c_1, c_2 \in \bR$ (note that
$c_2 \ne 0$). The sequence $\Lambda_0$ is separated if and only if
\begin{equation} \label{eq:aj.ne.bc}
 \alpha_1 \ne \left\{\frac{c_1 \beta_1 - (c_1^2 + c_2^2) \beta_2}{c_2}\right\}
 \quad\text{or}\quad \alpha_2 \ne \left\{\frac{\beta_1 - c_1 \beta_2}{c_2}\right\},
\end{equation}
where $\{x\}$ denotes the fractional part of $x \in \bR$.
\item[(iii)] For any $\eps>0$ there exists $\wt{C}_{\eps}>0$ such that
\begin{equation} \label{eq:Delta0>C.exp1.exp2}
 |\Delta_0(\l)| > \wt{C}_{\eps} \bigl(e^{-\Im(b_1 \l)} + 1\bigr)
 \bigl(e^{-\Im(b_2 \l)} + 1\bigr),
 \quad \l \in \bC \setminus \Omega_{\eps},
\end{equation}
where $\Omega_{\eps} := \bigcup_{\l_0 \in \Lambda_0} \bD_{\eps}(\l_0)$.
\end{itemize}
\end{lemma}
%
%
\begin{proof}
\textbf{(i)} It follows from~\eqref{eq:Delta0} and~\eqref{eq:dj=e.gammaj} that the
characteristic determinant $\Delta_0(\cdot)$ of the problem~\eqref{eq:riesz.system},
\eqref{eq:riesz.U1.U2.d1.d2} (with $Q = 0$) is given by
\begin{align}
\label{eq:Delta0=Delta01.Delta02}
 \Delta_0(\l) &= 1 + d_1 d_2 e^{i (b_1 + b_2) \l}
 - d_1 e^{i b_1 \l} - d_2 e^{i b_2 \l} \notag \\
 &= (d_1 e^{i b_1 \l} - 1) (d_2 e^{i b_2 \l} - 1)
 = \Delta_{01}(\l) \cdot \Delta_{02}(\l), \\
\label{eq:Delta0j.def}
 \Delta_{0j}(\l) &:= e^{i (b_j \l - 2 \pi \gamma_j)} - 1, \quad j \in \{1,2\}.
\end{align}
It is clear that the sequence $\Lambda_{0j} := \{\l_{nj}^0\}_{n \in \bZ}$ is the
sequence of the zeros of $\Delta_{0j}(\cdot)$, $j \in \{1, 2\}$. Hence
formulas~\eqref{eq:Lambda0} for zeros of $\Delta_0(\cdot)$ are immediate from
factorization~\eqref{eq:Delta0=Delta01.Delta02}.

\textbf{(ii)} The sequences $\Lambda_{0j} := \{\l_{nj}^0\}_{n \in \bZ}$, $j \in
\{1,2\}$, form two arithmetic progressioms. Hence the sequence $\Lambda_0$ is
separated if and only if the sequences $\Lambda_{01}$ and $\Lambda_{02}$ does
not have common entries. This in turn is reduced to solving the following
Diophantine equation
\begin{equation} \label{eq:ln1=lm2}
 \l_{n1}^0 = 2 \pi b_1^{-1} (\gamma_1 + n)
 = 2 \pi b_2^{-1} (\gamma_2 + m) = \l_{m2}^0, \quad n,m \in \bZ.
\end{equation}
Inserting formula $b_1 b_2^{-1} = c_1 + i c_2$ and representations
\eqref{eq:dj=e.gammaj} for $\gamma_1$, $\gamma_2$, in~\eqref{eq:ln1=lm2}, one
reduces this equation to
\begin{equation} \label{eq:a1n=c.a2m}
 \alpha_1 + i \beta_1 + n = (c_1 + i c_2) (\alpha_2 + i \beta_2 + m).
\end{equation}
Separating real and imaginary parts in~\eqref{eq:a1n=c.a2m} we arrive at the
system
\begin{equation} \label{eq:a1n.b1.system}
 \begin{cases}
 \alpha_1 + n = c_1 \alpha_2 - c_2 \beta_2 + c_1 m, \\
 \beta_1 = c_2 \alpha_2 + c_1 \beta_2 + c_2 m.
 \end{cases}
\end{equation}
Since $c_2 \ne 0$, the solution to this system is given by
\begin{equation} \label{eq:m.n.system}
 \begin{cases}
 m = \frac{\beta_1 - c_1 \beta_2}{c_2} - \alpha_2, \\
 n = \frac{c_1 \beta_1 - (c_1^2 + c_2^2) \beta_2}{c_2} - \alpha_1.
 \end{cases}
\end{equation}
Since $\alpha_1, \alpha_2 \in [0,1)$, both $m$ and $n$ in~\eqref{eq:m.n.system}
are integers if and only if both conditions~\eqref{eq:aj.ne.bc} are violated.
Hence Diophantine equation~\eqref{eq:ln1=lm2} does not have integer solutions if
and only if condition~\eqref{eq:aj.ne.bc} holds. This completes the proof.

\textbf{(iii)} Let $\Omega_{\eps}^0 := \bigcup_{n \in \bZ} \bD_{\eps}(\pi n)$.
We need the following well known estimate from below (see
e.g.~\cite[Theorem~I.5.7]{Leont83})
\begin{equation} \label{eq:sin.z>A.exp}
 |\sin z| > A_{\eps} \cdot e^{|\Im z|} > 2^{-1}A_{\eps} \cdot
 (e^{\Im z} + e^{-\Im z}), \quad z \in \bC \setminus \Omega_{\eps}^0.
\end{equation}
for some $A_{\eps} > 0$. Since $2 |\sin z| = |e^{2 i z} - 1| \cdot e^{\Im z}$,
it follows from~\eqref{eq:sin.z>A.exp} that
\begin{equation} \label{eq:exp.z-1>A.exp}
 |e^{2 i z} - 1| > A_{\eps} \cdot (e^{-2\Im z} + 1),
 \quad z \in \bC \setminus \Omega_{\eps}^0,
\end{equation}
Since the sequence $\Lambda_{0j} := \{\l_{nj}^0\}_{n \in \bZ}$ is the sequence
of zeros of $\Delta_{0j}(\cdot)$, estimate~\eqref{eq:exp.z-1>A.exp} yields the
following estimate from below on $|\Delta_{0j}(\cdot)|:$
\begin{align} \label{eq:Delta0j>A.exp}
 |\Delta_{0j}(\l)| &= \bigl|e^{i (b_j \l - 2 \pi \gamma_j)} - 1\bigr| > A_{\eps}
 \cdot \bigl(e^{\Im (2 \pi \gamma_j)} \cdot e^{-\Im (b_j \l)} + 1\bigr) \notag \\
 & > A_{\eps j} \cdot \bigl(e^{-\Im (b_j \l)} + 1\bigr),
 \quad \l \in \bC \setminus \bigcup_{n \in \bZ} \bD_{\eps}(\l_{nj}^0),
 \quad j \in \{1,2\},
\end{align}
for some $A_{\eps j} > 0$, $j \in \{1,2\}$. Since the sequence $\Lambda_{0j}$ is
the subsequence of $\Lambda_{0}$ for $j \in \{1,2\}$,
estimate~\eqref{eq:Delta0j>A.exp} holds automatically for $\l \in \bC \setminus
\Omega_{\eps}$. Inserting inequalities~\eqref{eq:Delta0j>A.exp} into
factorization identity~\eqref{eq:Delta0=Delta01.Delta02} yields the desired
estimate~\eqref{eq:Delta0>C.exp1.exp2} with $\wt{C}_{\eps} = A_{\eps 1}
A_{\eps 2}$.
\end{proof}
The following estimate will be of importance below.
%
%
\begin{corollary} \label{cor:Delta0.estim}
Let $\Delta_0(\cdot)$ and $\Omega_{\eps}$ be as in Lemma \ref{lem:delta0}. Then
for any $\eps>0$ there exists $C_{\eps}>0$ such that
\begin{equation} \label{eq:Delta0>C.exp}
 |\Delta_0(\l)| > C_{\eps} \bigl(e^{-\Im(b_1 \l)} + e^{-\Im(b_2 \l)} + 2\bigr),
 \quad \l \in \bC \setminus \Omega_{\eps}.
\end{equation}
\end{corollary}
%
%
\begin{proof}
This estimate is immediate from~\eqref{eq:Delta0>C.exp1.exp2} with $C_{\eps} =
\wt{C}_{\eps}/2$.
\end{proof}
In what follows we need the following version of the Riemann-Lebesgue lemma (see
e.g.~\cite[Lemma 3.5]{LunMal16JMAA}).
%
%
\begin{lemma} \label{lem:RimLeb}
Let $g \in L^1[0,1]$ and $c \in \bC \setminus \{0\}$. Then for any $\delta > 0$
there exists $M_{\delta} > 0$ such that
\begin{equation} \label{eq:int.g.e<delta}
 \left|\int_0^1 g(t) e^{i c \l t} dt \right| < \delta (e^{-\Im(c \l)} + 1),
 \quad |\l| > M_{\delta}.
\end{equation}
\end{lemma}
%
%
The following asymptotic formula for the eigenvalues of the problem
\eqref{eq:riesz.system},~\eqref{eq:riesz.BQ},~\eqref{eq:riesz.U1.U2.d1.d2} plays
a crucial role in the study of Riesz basis property.
%
%
\begin{proposition} \label{prop:Lambda.asymp}
Let $\Delta(\cdot)$ be the characteristic determinant of the problem
\eqref{eq:riesz.system}, \eqref{eq:riesz.BQ}, \eqref{eq:riesz.U1.U2.d1.d2} and
let $\Lambda_0 = \{\l_{nj}^0\}_{n \in \bZ, \, j \in \{1,2\}}$, be a sequence
given by~\eqref{eq:Lambda0}. Then the following statements hold:
\begin{itemize}
\item[(i)] The sequence $\Lambda$ of its zeros can be ordered as $\Lambda =
\{\l_{nj}\}_{n \in \bZ, j \in \{1,2\}}$ in such a way that the following
asymptotic formula holds
\begin{equation} \label{eq:l.n=l.n0+o(1)}
 \l_{nj} = \l_{nj}^0 + o(1), \quad\text{as}\quad n \to \infty, \quad n \in \bZ,
 \quad j \in \{1,2\}.
\end{equation}
In particular, the sequence $\Lambda$ is asymptotically separated.
\item[(ii)] Let in addition condition~\eqref{eq:aj.ne.bc} be satisfied. Then
there exists a constant $C > 0$ such that the sequence $\Lambda$ is separated
whenever $\|Q\|_{C[0,1]} < C$.
\end{itemize}
\end{proposition}
%
%
\begin{proof}
\textbf{(i)} Let $\Delta_0(\cdot)$ be the characteristic determinant of the
problem \eqref{eq:riesz.system},~\eqref{eq:riesz.U1.U2.d1.d2} with $Q = 0$ and
let $\eps > 0$. Combining Lemma~\ref{lem:Delta=Delta0+} with Lemma~\ref{lem:RimLeb}
yields the following estimate for any $\delta > 0$,
\begin{equation} \label{eq:Delta-Delta0}
 |\Delta(\l) - \Delta_0(\l)| < \delta \bigl(e^{-\Im (b_1 \l)}
 + e^{-\Im (b_2 \l)} + 2\bigr), \quad |\l| > M_{\delta},
\end{equation}
with certain constant $M_{\delta} > 0$. By Lemma~\ref{lem:delta0}, (iii), there
exists a constant $C_{\eps} > 0$ such that the estimate~\eqref{eq:Delta0>C.exp}
holds. Combining estimates~\eqref{eq:Delta0>C.exp} and~\eqref{eq:Delta-Delta0}
with $\delta = C_{\eps}$ we arrive at the following important estimate
\begin{align} \label{eq:Delta-Delta0<Delta0}
 |\Delta(\l) - \Delta_0(\l)| < |\Delta_0(\l)|, \quad
 \l \in \bC \setminus \wt{\Omega}_{\eps}, \quad
 \wt{\Omega}_{\eps} := \bD_{R_{\eps}} \cup \Omega_{\eps}, \quad
 R_{\eps} := M_{C_{\eps}}.
\end{align}
Since $\Omega_{\eps} = \bigcup_{\l_0 \in \Lambda_0} \bD_{\eps}(\l_0)$, where
$\Lambda_0$ is a sequence of zeros of the determinant $\Delta_0(\cdot)$,
estimate~\eqref{eq:Delta-Delta0<Delta0} makes it possible to apply the classical
Rouche theorem. Its applicability ensures that all zeros of the determinant
$\Delta(\cdot)$ lie in the domain $\wt{\Omega}_{\eps}$. Moreover, each connected
component of $\wt{\Omega}_{\eps}$ contains the same number of zeros of
determinants $\Delta(\cdot)$ and $\Delta_0(\cdot)$ counting multiplicity. Since
in accordance with Lemma~\ref{lem:delta0},~(i), the sequence of zeros $\Lambda_0$
is asymptotically separated, it follows that for $\eps$ sufficiently small,
discs $\bD_{\eps}(\l_0)$, $\l_0 \in \Lambda_0$, $|\l_0| > N_0$, do not intersect
each other, where $N_0$ does not depend on $\eps$. Hence $\wt{\Omega}_{\eps}
\setminus \bD_{N_{\eps}}$, $N_{\eps} := \max\{R_{\eps}, N_0\}$, is a union of
disjoint disc parts $\bD_{\eps}(\l_0) \setminus \bD_{N_{\eps}}$, $\l_0 \in
\Lambda_0$, $|\l_0| > N_{\eps} - \eps$. Thus, each of these disc parts contains
exactly one (simple) zero of $\Delta(\cdot)$. Since $\eps > 0$ is arbitrary small,
the latter implies the desired asymptotic formula~\eqref{eq:l.n=l.n0+o(1)} and
also that the sequence $\Lambda$ is asymptotically separated.

\textbf{(ii)} By Lemma~\ref{lem:delta0},~(ii), condition~\eqref{eq:aj.ne.bc}
ensures that the sequence of zeros $\Lambda_0$ is separated. Let $\eps > 0$ be
such that all the discs $\bD_{\eps}(\l_0)$, $\l_0 \in \Lambda_0$, are disjoint.
By Corollary~\ref{cor:Delta0.estim}, there exists $C_{\eps} > 0$ such that
estimate from below~\eqref{eq:Delta0>C.exp} holds. Choose $C > 0$ so small that
$C_0 C \cdot \exp(C_1 C) < C_{\eps}$, where $C_0, C_1$ are the constants from
inequality~\eqref{eq:gj.in.C}. Assuming that $\|Q\| = \|Q\|_{C[0,1]} < C$ one
easily gets from~\eqref{eq:gj.in.C} that
\begin{equation} \label{eq:gjt<C}
 |g_j(t)| < C_0 C \cdot \exp(C_1 C) < C_{\eps}, \quad t \in [0,1],
 \quad j \in \{1, 2\}.
\end{equation}
Further, it is clear that
\begin{equation} \label{eq:ebjt<}
 \left|e^{i b_j \l t}\right| < e^{- \Im(b_j \l)} + 1, \quad t \in [0,1],
 \quad \l \in \bC, \quad j \in \{1, 2\}.
\end{equation}
Combining formula~\eqref{eq:Delta=Delta0+} with estimates~\eqref{eq:gjt<C},
\eqref{eq:ebjt<} and~\eqref{eq:Delta0>C.exp} we arrive at
\begin{align} \label{eq:Delta-Delta0.sep}
 |\Delta(\l) - \Delta_0(\l)|
 & \le \int^1_0 |g_1(t)| \cdot \left|e^{i b_1 \l t}\right| dt
 + \int^1_0 |g_2(t)| \cdot \left|e^{i b_2 \l t}\right| dt \notag \\
 & < C_{\eps} \bigl(e^{-\Im (b_1 \l)} + e^{-\Im (b_2 \l)} + 2\bigr)
  < |\Delta_0(\l)|,
 \quad \l \in \bC \setminus \Omega_{\eps}^0.
\end{align}
Since the discs $\bD_{\eps}(\l_0)$, $\l_0 \in \Lambda_0$, are disjoint, the
Rouche theorem now implies that each of these discs contains exactly one (simple)
zero of $\Delta(\cdot)$, which implies that the sequence $\Lambda$ is separated.
\end{proof}
Next we recall classical definitions of Riesz basisness and Riesz basisness with
parentheses (see e.g.~\cite{GohKre65} and~\cite{Markus86}).
%
%
\begin{definition} \label{def:basis}
(i) A sequence $\{f_k\}_{k=1}^{\infty}$ of vectors in $\fH$ is called a
\textbf{Riesz basis} if it admits a representation $f_k = T e_k$, $k \in \bN$,
where $\{e_k\}_{k=1}^{\infty}$ is an orthonormal basis in $\fH$ and $T : \fH \to
\fH$ is a bounded operator with bounded inverse.

(ii) A sequence of subspaces $\{\fH_k\}_{k=1}^{\infty}$ is called a \textbf{Riesz
basis of subspaces} in $\fH$ if there exists a complete sequence of mutually
orthogonal subspaces $\{\fH'_k\}_{k=1}^{\infty}$ and a bounded operator $T$ in
$\fH$ with bounded inverse such that $\fH_k = T \fH'_k$, $k \in \bN$.

(iii) A sequence $\{f_k\}_{k=1}^{\infty}$ of vectors in $\fH$ is called a
\textbf{Riesz basis with parentheses} if each its finite subsequence is linearly
independent, and there exists an increasing sequence $\{n_k\}_{k=0}^{\infty}
\subset \bN$ such that $n_0=1$ and the sequence $\fH_k := \Span\{f_j\}_
{j=n_{k-1}}^{n_k-1}$, forms a Riesz basis of subspaces in $\fH$. Subspaces
$\fH_k$ are called blocks.
\end{definition}
%
%
Note that if $A$ is an operator in $\fH$ with discrete spectrum, then the
property of its root vectors (eigenvectors) to form a Riesz basis with
parentheses (Riesz basis) in $\fH$ can be retranslated in terms of $A$ to be
close to a certain "good" operator.

To retranslate this property we recall that an eigenvalue $\l_0$ of an operator
$A$ is called algebraically simple if $\ker (A - \l_0) = \cR(\l_0, A)$, where
$\cR(\l_0, A)$ is the root subspace of $A$. It is equivalent to the fact that
$\l_0$ is a simple pole (i.e. pole of order one) of the resolvent $(A - \l)^{-1}$.
In finite dimensional case this means that all Jordan cells corresponding to
$\l_0$ are of size one.

In applications to BVP it is convenient to reformulate Riesz basis property
(with and without parentheses) of systems of root functions of operators with
discrete spectrum in terms of their similarity to certain subclasses of the
class of \emph{spectral operators}. Recall that an operator is called \emph{a
spectral operator} if it admits a countably additive (generally non-orthogonal)
resolution of identity defined on Borel subsets of complex plane.

Next we collect several definitions in a form most suitable for the purposes of
our paper.
%
%
\begin{definition} \label{def:spectral.operator}
(i) A bounded operator $N$ in $\fH$ is called \textbf{quasi-nilpotent} if
$\sigma(N) = \{0\}$.

(ii) A closed operator $S$ in $\fH$ is called an \textbf{operator of scalar type}
if it is similar to a normal operator.

(iii) A closed operator $A$ in $\fH$ will be called \textbf{almost normal} if it
admits an orthogonal decomposition $A = A_1 \oplus A_2$ where $A_1$ is finite
dimensional and $A_2$ is normal.
\end{definition}
%
%
Note that according to~\cite[Theorem~XVIII.2.28]{DunSchw71} the operator $T = S
+ N$, where $S$ is an operator of scalar type and $N$ is a quasi-nilpotent
operator that commutes with $S$, is spectral operator. Such a representation
\emph{characterizes bounded spectral operators}, while becomes false, in general,
for unbounded operators (see~\cite[Section~XVIII.2]{DunSchw71}).

Note also that the definition of a scalar type operator is given in accordance
with the Wermer theorem, see~\cite[Theorem XV.6.4]{DunSchw71}.
%
%
\begin{lemma} \label{lem:similarity_to_normal}
Let $A$ be a closed densely defined operator in $\fH$ with discrete spectrum.
Then

(i) For the operator $A$ to be similar to a normal operator in $\fH$ it is
necessary and sufficient that its eigenvalues are algebraically simple and its
system of eigenvectors $\{f_k\}_{k\in \bN}$ forms a Riesz basis in $\fH$.

(ii) For the operator $A$ to be similar to an almost normal operator in $\fH$
it is necessary and sufficient that all its eigenvalues but finitely many are
algebraically simple and its system of eigen- and associated vectors
$\{f_k\}_{k\in \bN}$ forms a Riesz basis in $\fH$.

(iii) The operator $A$ is similar to an orthogonal direct sum of finite dimensional
operators if and only if the system of its root vectors forms a Riesz basis with
parentheses. In particular, in this case $A$ is a spectral operator.
\end{lemma}
%
%
\begin{proof}
(i) \emph{The necessity} is obvious because of the corresponding properties of
normal operator and Definition \ref{def:basis}(i).

\emph{To prove sufficiency} let $\{\l_k\}_{k \in \bN}$ be a sequence of
eigenvalues of $A$, counting (geometric) multiplicities, $Af_k = \l_k f_k$, $k
\in \bN$. Since $\{f_k\}_{k=1}^{\infty}$ forms a Riesz basis in $\fH$, there
exists a bounded operator $T : \fH \to \fH$ with bounded inverse and an
orthonormal basis $\{e_k\}_{k=1}^{\infty}$ in $\fH$ such that $f_k = T e_k$,
$k \in \bN$.

Define a diagonal operator $S$ in $\fH$ by setting $S e_k = \l_k e_k$, $k \in \bN$,
and extending  it by linearity to the natural (maximal) domain of definition.
Clearly, $S$ is normal and $T^{-1} A T e_k = S e_k$, $k \in \bN$, i.e.  $A$ and
$S$ are similar.

(ii) and (iii) are proved similarly if one defines an operator $S$ accordingly in
an orthogonal Jordan chain chosen in each (necessarily finite-dimensional) root
subspace $\cR(\l_0, A)$ corresponding to each eigenvalue $\l_0$ of $A$.

Clearly, an orthogonal sum of finite dimensional operators admits the
representation $S + N$, where $S$ is normal operator and $N$ is a quasi-nilpotent
operator that commutes with $S$. Hence operator $A$ in part (iii) is similar to a
spectral operator according to~\cite[Theorem~XVIII.2.28]{DunSchw71} and thus is
spectral operator itself.
\end{proof}
To prove the main result of the Section let us recall a result from~\cite{LunMal15}
on Riesz basis property with parentheses for the operator $L_{U_1,U_2}(B,Q)$.
%
%
\begin{proposition}~\cite[Proposition 5.9]{LunMal15} \label{prop:basis.per}
Let $Q \in L^{\infty}([0,1]; \bC^{2 \times 2})$ and let $L := L_{U_1,U_2}(B,Q)$
be an operator associated with the BVP~\eqref{eq:riesz.system},
\eqref{eq:riesz.U1.U2.d1.d2}. Then the system of root functions of $L$ forms a
Riesz basis with parentheses in $L^2([0,1]; \bC^2)$.
\end{proposition}
%
%
Finally, we are ready to state the main result of the section on Riesz basis
property of the operator $L_{U_1,U_2}(B,Q)$.
%
%
\begin{theorem} \label{th:similarity-to_normal}
Let $Q \in A(\bD_R; \bC^{2 \times 2})$, $b_1 b_2^{-1} \notin \bR$ and let $L :=
L_{U_1,U_2}(B,Q)$ be the operator associated with the BVP~\eqref{eq:riesz.system},
\eqref{eq:riesz.U1.U2.d1.d2}. Then the following statements hold:
\begin{itemize}
\item[(i)] The operator $L$ is similar to an almost normal operator. In
particular, each eigenvalue of $L$ but finitely many is algebraically simple and
the system of root functions of $L$ forms a Riesz basis in $L^2([0,1]; \bC^2)$.
\item[(ii)] Let in addition condition~\eqref{eq:aj.ne.bc} holds. Then there
exists $C > 0$ such that the operator $L$ is similar to a normal operator
whenever $\|Q\|_{C[0,1]} < C$.
\end{itemize}
\end{theorem}
%
%
\begin{proof}
\textbf{(i)} Due to~\cite[Proposition 5.9]{LunMal15} the system of root functions
of the operator $L$ forms a Riesz basis with parentheses in $L^2([0,1]; \bC^2)$,
where each block is constituted by the root subspaces corresponding to the
eigenvalues of $L$ that are mutually $\eps$-close with respect to the sequence
$\Psi := \{-\varphi_1, -\varphi_2, \pi - \varphi_1, \pi - \varphi_2\}$. Here
$\varphi_j = \arg b_j$, $j \in \{1, 2\}$, and $\eps > 0$ is sufficiently small.
Recall that numbers $\l, \mu \in \bC$ are called $\eps$-close with respect to
the sequence $\{\psi_k\}_{k=1}^m$ if for some $k \in \{1, \ldots, m\}$ they
belong to a small angle of size $2\eps$ with the bisectrix $l_+(\psi_k) := \{\l
\in \bC : \arg \l = \psi_k\}$ and their projections on this ray differ no more
than by $\eps$ (see \cite[Definition~5.4]{LunMal15})

Let us prove that for $\eps (> 0)$ sufficiently small the above blocks are
asymptotically of size one, i.e. $n_{k+1} = n_k + 1$ for sufficiently large $k$
(see Definition~\ref{def:basis}, (iii)). Due to Proposition~\ref{prop:Lambda.asymp},
(i), eigenvalues of $L$ are asymptotically simple and separated and, due to
asymptotic formula~\eqref{eq:l.n=l.n0+o(1)} and the form~\eqref{eq:Lambda0} of
the sequence $\Lambda_0 = \{\l_{nj}^0\}_{n \in \bZ, \, j \in \{1, 2\}}$, they
are located along 2 different non-parallel lines of $\bC$ that are parallel to
the rays $l_+(-\varphi_1)$, $l_+(-\varphi_2)$. It is clear now that for
sufficiently small $\eps > 0$ and sufficiently large $n, m \in \bZ$ different
numbers $\l_{nj}$ and $\l_{mk}$ are not $\eps$-close with respect to the
sequence $\Psi$, for $j, k \in \{1, 2\}$. Indeed, if $j \ne k$ then they don't
belong to a small angle with the bisectrix $l_+(\psi)$ for any $\psi \in \Psi$,
since they are close to 2 different non-parallel lines of $\bC$. If $j=k$ then
$n \ne m$ and numbers $\l_{nj}$ and $\l_{mj}$ belong to a small angle with the\
bisectrix $l_+(\psi)$ for some $\psi \in \Psi$. From the form~\eqref{eq:Lambda0}
of the sequence $\Lambda_0$ and asymptotic formula~\eqref{eq:l.n=l.n0+o(1)} it
is clear that projections of $\l_{nj}$ and $\l_{mj}$ on $l_+(\psi)$ are
separated.

Thus, $n_{k+1} = n_k + 1$ for sufficiently large $k$, hence the system of root
functions of $L$ forms a Riesz basis (without parentheses) in $L^2([0,1]; \bC^2)$.

\textbf{(ii)} By \textbf{(i)} the system of root functions of the operator $L$
forms a Riesz basis in $L^2([0,1]; \bC^2)$. On the other hand, in accordance
with Proposition~\ref{prop:Lambda.asymp}, (ii), eigenvalues of $L$ are
(algebraically and geometrically) simple and separated provided that $\|Q\| <
C$, for certain $C > 0$. To get a similarity of $L$ to a normal operator it
remains to apply Lemma~\ref{lem:similarity_to_normal}, (i).
\end{proof}
%
%
\begin{remark}\label{rem:Riesz_bas_prop_for_Dirac}
(i) The Riesz basis property for $2 \times 2$ Dirac operators $L_{U_1, U_2}$
has been investigated in numerous papers (see~\cite{TroYam02, DjaMit10BariDir,
Bask11, DjaMit12UncDir, DjaMit12Crit, DjaMit13CritDir, Gub03, LunMal14Dokl,
LunMal16JMAA, SavShk14} and references therein). The most complete result was
recently obtained independently and by different methods in~\cite{LunMal14Dokl,
LunMal16JMAA} and~\cite{SavShk14}. Namely, assuming that $B=B^*$ and $Q(\cdot)
\in L^1([0,1]; \bC^{2 \times 2})$ it is proved in~\cite{LunMal14Dokl,
LunMal16JMAA} (the general case of $b_1b_2 <0$) and~\cite{SavShk14} (the Dirac
case, $b_1= -b_2$) that the system of root vectors of equation~\eqref{eq:system}
subject to \emph{regular boundary conditions} constitutes a \emph{Riesz basis
with parentheses} in $L^2([0,1]; \bC^2)$ and ordinary Riesz basis provided that
BC are \emph{strictly regular}.

(ii) Note also that an important role in proving Riesz basis property
in~\cite{LunMal14Dokl, LunMal16JMAA} and~\cite{SavShk14} is playing the
following asymptotic formula
\begin{equation} \label{eq:l.n=l.n0+o(1)_Intro}
	\l_n = \l_n^0 + o(1), \quad\text{as}\quad n \to \infty, \quad n \in \bZ,
\end{equation}
for the eigenvalues $\{\l_n\}_{n \in \bZ}$ of the operator $L_{C, D}(B, Q)$ with
regular BC (and summable potential matrix $Q$), where $\{\l_n^0\}_{n \in \bZ}$
is the sequence of eigenvalues of the unperturbed operator $L_{C, D}(B, 0)$.
Note also that formula~\eqref{eq:l.n=l.n0+o(1)_Intro} has recently been applied
to investigation of spectral properties of Dirac systems on star
graphs~\cite{AdamLang16}.

(iii) Note that the periodic problem for system
\eqref{eq:riesz.system}--\eqref{eq:riesz.BQ} substantially differs from that
for Dirac operators. Namely, periodic BVP for system
\eqref{eq:riesz.system}--\eqref{eq:riesz.BQ} is always strictly regular, while
for Dirac system it is only regular.

Another proof of Theorem \ref{th:similarity-to_normal}(i) can also be obtained
in just the same way as the proof of Riesz basis property for Dirac operators
in~\cite{LunMal14Dokl, LunMal16JMAA}. The proof ignores Proposition
\ref{prop:basis.per} and is completely relied on transformation operators.
\end{remark}
%
%
\begin{remark}
Numerous papers are devoted to the completeness and Riesz basis property for the
Sturm-Liouville operator (see the recent surveys~\cite{Mak12, Mak13, Mak15} by
A.S.~Makin and the papers cited therein). In connection with Theorem
\ref{th:similarity-to_normal} we especially mention the recent achievements for
periodic (anti-periodic) Sturm-Liouville operator $-\frac{d^2}{dx^2} + q(x)$ on
$[0,\pi]$. Namely, F.~Gesztesy and V.A.~Tkachenko~\cite{GesTka09,GesTka12} for
$q \in L^2[0,\pi]$ and P.~Djakov and B.S.~Mityagin~\cite{DjaMit12Crit} for $q
\in W^{-1,2}[0,\pi]$ established by different methods a \emph{criterion} for the
system of root functions to contain a Riesz basis.
\end{remark}
%
%
\section{Completeness property under rank one perturbations}
\label{sec:rank.one}
%
%
First we recall definition of a dual pair.
%
%
\begin{definition}
\begin{itemize}
\item[(i)] A pair $\{S_1,S_2\}$ of closed densely defined operators in $\fH$ is
called a dual pair of operators if $S_1 \subset S_2^* \ (\Longleftrightarrow
S_2\subset S_1^*)$.
\item[(ii)]
An operator $T$ is called a proper extension of the dual pair $\{S_1,S_2\}$ if
$S_1\subset T\subset S_2^*$.
\end{itemize}
\end{definition}
%
%
\begin{example}
A typical example one obtains by choosing $\{S_1,S_2\}$ to be the minimal
operators associated in $\fH = L^2[0,1]$ with Sturm-Liouville differential
expressions $\cL(q) = -d^2/dx^2 + q$ and $\cL(\ol{q}) = -d^2/dx^2 + \ol{q}$,
respectively. Assuming that $q\in L^2[0,1]$ one gets that the minimal operators
$L_{\min}(q)$ and $L_{\min}(\ol{q})$ are given by differential expressions
$\cL(q)$ and $\cL(\ol{q})$ on the domain
\begin{equation*}
 \dom(L_{\min}(q)) = \dom(L_{\min}(\ol{q})) = W^{2,2}_{0}[0,1].
\end{equation*}
\end{example}
%
%
\begin{example}
Another example one obtains by choosing $S_1$ and $S_2$ to be the minimal
operators associated in $L^2([0,1]; \bC^n)$ with expression~\eqref{eq:system.nxn}
and its formal adjoint ${\cL}(B^*, Q^*) = -i \(B^*\)^{-1} d/dx + Q^*(x)$:
\begin{equation} \label{eq:dual_pair_of_L(B,Q)}
 S_1 := L_{\min}(B,Q) \quad \text{and} \quad S_2 := L_{\min}(B^*,Q^*).
\end{equation}
If $Q(\cdot) \in L^2([0,1]; \bC^{n\times n})$, then
\begin{equation*}
 \dom(S_1) = \dom(S_2) = W^{1,2}_{0}([0,1];\bC^n).
\end{equation*}
%
%
%
\end{example}
%
%
Next we assume that $n=2$ and that matrices $B$ and $Q(\cdot)$ are given by
\begin{equation}\label{eq:normal.BQ}
 B = \text{diag}(b_1, b_2),
 \quad b_1 b_2^{-1} \not \in \bR, \quad\text{and}\quad
 Q = \begin{pmatrix} 0 & Q_{12} \\ Q_{21} & 0 \end{pmatrix}
 \in L^1([0,1]; \bC^{2\times 2}).
\end{equation}
Denote by $L_{U_1,U_2}(B,Q)$ the operator generated by the equation
\begin{equation}\label{eq:normal.2x2}
 \cL y := -i B^{-1} y' + Q(x)y = \l y, \qquad y = \col(y_1,y_2), \qquad
 x \in [0,1],
\end{equation}
subject to the boundary conditions
\begin{equation}\label{eq:BC.2x2}
	U_j(y) := a_{j 1}y_1(0) + a_{j 2}y_2(0) + a_{j 3}y_1(1) + a_{j 4}y_2(1) = 0,
	\quad j \in \{1,2\}.
\end{equation}

First, we recall following~\cite{LunMal18} a description of normal extensions
of the dual pair $\{S_1,S_2\}$ of the form~\eqref{eq:dual_pair_of_L(B,Q)}, i.e.
all normal operators $L_{U_1,U_2}(B,Q)$ generated by the BVP
\eqref{eq:normal.2x2}--\eqref{eq:BC.2x2}.
%
%
\begin{lemma} \cite{LunMal18} \label{lem:normal.BC}
Operator $L_{U_1,U_2}(B,0)$ is normal if and only if boundary conditions
$\{U_1, U_2\}$ are of the form
\begin{equation} \label{eq:U1.U2.d1.d2}
 U_1(y) = y_1(0) - d_1 y_1(1) = 0, \quad
 U_2(y) = y_2(0) - d_2 y_2(1) = 0, \quad |d_1| = |d_2| = 1.
\end{equation}
\end{lemma}
%
%
\begin{proposition} \cite{LunMal18} \label{prop:normal=const}
Let $B$ and $Q$ be given by~\eqref{eq:normal.BQ}, $Q \in
L^1\([0,1];\bC^{2\times 2}\)$, $Q \not \equiv 0$. Then the operator
$L_{U_1,U_2}(B,Q)$ is normal if and only if a potential matrix $Q(\cdot)$ is a
constant matrix of the form
\begin{equation}\label{Q=const}
 Q(x) = \(b_1^{-1} - b_2^{-1}\) \begin{pmatrix}
 0 & q \\
 \ol{q} & 0
 \end{pmatrix} \ne Q^*(x), \quad x\in [0,1],
 \quad q \in \bC \setminus \{0\},
\end{equation}
and boundary conditions~\eqref{eq:BC.2x2} are of the form
\begin{equation} \label{eq:y1=e.y0}
 y(1) = e^{i \varphi} y(0), \quad \varphi \in [-\pi, \pi).
\end{equation}
\end{proposition}
%
%
Alongside operator $L_{U_1,U_2}(B,Q)$ we consider operator $L_{\wt{U}_1,
\wt{U}_2}(B,Q)$, generated by equation~\eqref{eq:normal.2x2} subject to the
following boundary conditions
\begin{equation} \label{eq:U1.U2.h0.h1.again}
 \wt{U}_1(y):=y_1(0)-h_1 y_2(0)=0, \quad \wt{U}_2(y)= y_1(1)-h_2 y_2(0)=0,
 \quad h_1 h_2 \ne 0.
\end{equation}
Clearly, $\wt{J}_{14} = \wt{J}_{34} = 0$, hence boundary
conditions~\eqref{eq:U1.U2.h0.h1.again} are not weakly regular (cf.
relation~\eqref{cond:w-reg}).

To prove the main result we recall a completeness result from~\cite{MalOri12}.
%
%
\begin{theorem}[\cite{MalOri12}, Theorem 6.1]\label{th61}
Let $B$ and $Q(\cdot)$ be given by~\eqref{eq:normal.BQ}. Then the system of root
vectors of the operator $L_{\wt{U}_1,\wt{U}_2}(B,Q)$ generated by the BVP
\eqref{eq:normal.2x2},~\eqref{eq:U1.U2.h0.h1.again}, is complete and minimal in
$L^2([0,1];\bC^2)$.
\end{theorem}
%
%
Next we show that under certain additional assumption on a potential matrix, the
adjoint operator $\bigl(L_{\wt{U}_1,\wt{U}_2}(B,Q)\bigr)^*$ may be incomplete.
In particular, it is true in the case of trivial potential matrix $Q \equiv 0$.
%
%
\begin{proposition} \label{prop:not.complete}
Let $B$ and $Q(\cdot)$ be given by~\eqref{eq:normal.BQ}. Assume also that
$Q_{12}(\cdot)$ vanishes at the neighborhood of the endpoint 1, i.e. for some
$a \in (0,1)$
\begin{equation} \label{eq:Q12=0}
 Q_{12}(x) = 0, \quad \text{for a.e.} \quad x \in [a, 1].
\end{equation}
Then the operator $L_{\wt{U}_1,\wt{U}_2}(B,Q)$ corresponding to the BVP
\eqref{eq:normal.2x2},~\eqref{eq:U1.U2.h0.h1.again}, is peculiarly complete
(cf.~Definition~\ref{def:peculiar}).
\end{proposition}
%
%
\begin{proof}
Completeness of $L_{\wt{U}_1,\wt{U}_2}(B,Q)$ is implied by Theorem~\ref{th61}.
Let us verify that the adjoint operator is not complete. One easily checks that
$(L_{\wt{U}_1,\wt{U}_2})^*$ is given by the differential expression
\begin{equation} \label{eq:L*.expr}
 \cL^* y := -i B^{-*} y' + Q^*(x) y = \l y,
\end{equation}
and boundary conditions
\begin{equation} \label{eq:wtU*12}
 \wt{U}_{*,1}y=\ol{h_1} y_1(0)+\ol{b_1 b_2^{-1}} y_2(0) - \ol{h_2} y_1(1) = 0,
 \quad \wt{U}_{*,2}y = y_2(1) = 0,
\end{equation}
i. e. $L_{\wt{U}_1,\wt{U}_2}^* := (L_{\wt{U}_1,\wt{U}_2})^* =
L_{\wt{U}_{*,1}, \wt{U}_{*,2}}(B^*,Q^*)$.

Let $\l \in \sigma \bigl(L_{\wt{U}_1,\wt{U}_2}^*\bigr) = \{\l_j\}_1^{\infty}$ and
let $f = \binom{f_1}{f_2} \in \ker\bigl(L_{\wt{U}_1,\wt{U}_2}^* - \l\bigr)$ be
the corresponding eigenvector. Then the equation
\begin{equation*}
 \bigl(L_{\wt{U}_1,\wt{U}_2}^* - \l\bigr) f = 0
\end{equation*}
splits into the following system
\begin{equation} \label{eq:L*.eigen.system}
 \begin{cases}
 -i \ol{b_1^{-1}} f_1'(x) + \ol{Q_{21}(x)} f_2(x) = \l f_1(x), \\
 -i \ol{b_2^{-1}} f_2'(x) + \ol{Q_{12}(x)} f_1(x) = \l f_2(x).
 \end{cases}
\end{equation}
Since $Q_{12}(x) = 0$ for $x \in [a,1]$, the solution of the second equation
in~\eqref{eq:L*.eigen.system} on the interval $[a,1]$ is $f_2(x) =
C_2 e^{i\ol{b_2}\l x}$. The second boundary condition in~\eqref{eq:wtU*12}
implies that $C_2=0$, hence $f_2(x) = 0$ for $x \in [a,1]$. Inserting this
relation in the first of the equations in~\eqref{eq:L*.eigen.system} we conclude
that a solution of this system on the interval $[a,1]$ is proportional to a vector
\begin{equation}
 f(x) = \binom{e^{i\ol{b_1} \l x}}{0}, \quad x \in [a,1].
\end{equation}
Hence the system of eigenfunctions $\{u_j(\cdot)\}_{j=1}^{\infty}$ of the problem
\eqref{eq:L*.expr}--\eqref{eq:wtU*12} on the interval $[a,1]$ reads as follows
\begin{equation*}
 u_j(x):= \binom{u_{1j}(x)}{u_{2j}(x)} = \binom{e^{i\ol{b_1}\l_j x}}{0},
 \quad \l_j \in \sigma \bigl(L_{\wt{U}_1,\wt{U}_2}^*\bigr).
\end{equation*}
Therefore each vector $\binom{0}{g}\in L^2([0,1]; \bC^2)$ with $g$ satisfying
$\supp g \subset [a,1]$ is orthogonal to the system $\{u_j\}_1^{\infty}$. Thus,
the system $\{u_j\}_1^{\infty}$ is not complete in $L^2([0,1]; \bC^2)$ and its
orthogonal complement in $L^2([0,1]; \bC^2)$ is infinite dimensional.
%
%
\end{proof}
%
%
\begin{remark} \label{rem:alter.proof}
Here we show that incompleteness property of the adjoint operator in Proposition
\ref{prop:not.complete} can also be extracted from~\cite[Corollary 4.7]{LunMal15}.
Let us recall that in the case of $n=2$ it states that if one of the boundary
conditions is of the form $y_1(0)=0$ and $Q_{12}$ vanishes at the neighborhood of
0, then the system of root vectors of the operator $L_{\wt{U}_1, \wt{U}_2}(B,Q)$
is incomplete and its span is of infinite codimension in $L^2([0,1]; \bC^2)$.

Applying trivial linear transformations $i_1: \binom{y_1}{y_2} \mapsto
\binom{y_2}{y_1}$ and $i_2: y(x)\mapsto y(1-x)$ to the operator $L_{\wt{U}_{*,1},
\wt{U}_{*,2}}(B^*,Q^*)$ we reduce it to the operator $L_{\wh{U}_{*,1},
\wh{U}_{*,2}}(B^*,\wh{Q})$ where the new boundary condition $\wh{U}_{*,2}$ takes
the form $y_1(0)=0$ and $\wh{Q}_{12}(x) = \ol{Q_{12}(1-x)}$. Hence $\wh{Q}_{12}$
vanishes at the neighborhood of 0. Now incompleteness property of $L_{\wh{U}_{*,1},
\wh{U}_{*,2}}(B^*,\wh{Q})$, and hence of the operator $L_{\wt{U}_{*,1},
\wt{U}_{*,2}}(B^*,Q^*)$ is implied by~\cite[Corollary 4.7]{LunMal15}.
\end{remark}
%
%
Now we are ready to prove our first main result, Theorem
\ref{th:one.dim.perturb.normal}. It describes all pairs of operators $\{T,
\wt{S}\}$ giving an affirmative solution to \textbf{Problem 1} for the dual pair
$\{L_{\min}(B,Q), L_{\min}(B^*,Q^*) \}$ of minimal first order differential
operators admitting normal extensions $L_{U_1,U_2}(B,Q)$. It happen, in
particular, that for such pair of operators to exist \emph{a potential matrix
$Q$ is necessarily zero}.
\begin{proof}[Proof of Theorem~\ref{th:one.dim.perturb.normal}]
We need to prove that a pair of operators $\{T, \wt{S}\}$ is peculiar (i.e. $T
:= L_{U_1,U_2}(B,Q)$ is normal and $\wt{S} = L_{\wt{U}_1,\wt{U}_2}(B,Q)$ is
peculiarly complete) if and only if $Q \equiv 0$ and boundary conditions $\{U_1,
U_2\}$ and $\{\wt{U}_1, \wt{U}_2\}$ are equivalent to~\eqref{eq:U1.U2.d1.d2.intro}
and to~\eqref{eq:wt.U1.U2.d1.h1.intro}, respectively, in terms of
definition~\ref{def:bc.equiv}. Note, that~\eqref{eq:U1.U2.d1.d2.intro} is
identical to~\eqref{eq:U1.U2.d1.d2} and~\eqref{eq:wt.U1.U2.d1.h1.intro} is
identical to~\eqref{eq:U1.U2.h0.h1.again}.

\textbf{(i) Necessity.} Let $T = L_{U_1,U_2}(B,Q)$ be a normal operator and
$\wt{S} = L_{\wt{U}_1,\wt{U}_2}(B,Q)$ be peculiarly complete. Assume that $Q \not
\equiv 0$. Then by Proposition~\ref{prop:normal=const} matrix-function $Q$ is
constant of the form~\eqref{Q=const}. Thus, $Q$ is an entire matrix function and
$Q_{12}(0) Q_{12}(1) Q_{21}(0)Q_{21}(1) \ne 0$. Therefore in accordance
with~\cite[Corollary 1.7]{AgiMalOri12} this implies completeness of the system
of root vectors of the operator $L_V := L_{V_1, V_2}(B,Q)$ unless boundary
conditions $V_1(y)=V_2(y)=0$ represent an initial value problem ($y(0) = 0$ or
$y(1) = 0$). Clearly in the case of initial value problem both operators $L_V$
and $L_V^*$ have no spectrum. Since $\wt{S}$ is a complete operator, it follows
that boundary conditions $\{\wt{U}_1,\wt{U}_2\}$ does not represent initial
value problem. Clearly BVP corresponding to adjoint operator $\wt{S}^*$ is also
not an initial value problem. Hence $\wt{S}^*$ is also complete which
contradicts our assumption. Thus, $Q \equiv 0$. Equivalence of boundary conditions
$\{U_1, U_2\}$ to conditions~\eqref{eq:U1.U2.d1.d2} is now implied by
Lemma~\ref{lem:normal.BC}.

Next we investigate boundary conditions $\{\wt{U}_1,\wt{U}_2\}$ generated the
operator $\wt{S}$. If they are weakly $B$-regular (see
Definition~\ref{def:weakly.regular}) then, by Theorem \ref{th2.1_MO}, both
operators $\wt{S}$ and $\wt{S}^*$ are complete. Thus, boundary conditions
$\{\wt{U}_1,\wt{U}_2\}$ are not weakly $B$-regular. By~\cite[Lemma 2.7]{AgiMalOri12}
it means that they are either equivalent to BC~\eqref{eq:U1.U2.h0.h1.again} with
$h_1 h_2 \ne 0$ or to the boundary conditions
\begin{equation*}
 y_2(1) = 0, \quad \text{and} \quad a_{21} y_1(0) + a_{22} y_2(0) + a_{23} y_1(1) = 0.
\end{equation*}
In the latter case Proposition~\ref{prop:not.complete} implies that the system
of root vectors of the corresponding BVP is not complete. Thus, boundary
conditions of the operator $\wt{S} = L_{\wt{U}_1,\wt{U}_2}(B,Q)$ are necessarily
equivalent to boundary conditions~\eqref{eq:U1.U2.h0.h1.again}.

\textbf{(ii) Sufficiency.} Let $Q \equiv 0$ and BC of the operators $T =
L_{U_1,U_2}(B,Q)$ and $\wt{S} = L_{\wt{U}_1,\wt{U}_2}(B,Q)$ are equivalent
to~\eqref{eq:U1.U2.d1.d2} and~\eqref{eq:U1.U2.h0.h1.again}, respectively. Since
$Q$ vanishes at 0 and 1, Proposition~\ref{prop:not.complete} yields that $\wt{S}$
is peculiarly complete. Normality of the operator $T = L_{C,D}(B,0)$ follows from
Corollary~\ref{lem:normal.BC}. Clearly, resolvent difference $(\wt{S} - \l)^{-1}
- (T - \l)^{-1}$ is at most two-dimensional. Thus, $\{T, \wt{S}\}$ is a peculiar
pair.

Finally, we investigate the rank of the resolvent difference $(\wt{S} - \l)^{-1}
- (T - \l)^{-1}$ for the operators $\wt{S}$ and $T$ generated by boundary
conditions~\eqref{eq:U1.U2.h0.h1.again} and~\eqref{eq:U1.U2.d1.d2}, respectively.
Let $J_{jk}$ be determinants given by~\eqref{eq:Ajk.Jjk} for linear forms
\eqref{eq:U1.U2.d1.d2}. It follows from~\eqref{eq:U1.U2.d1.d2} that $J_{42}=0$,
$J_{14}=-d_2$ and $J_{34}=d_1 d_2$. Hence, condition~\eqref{eq:h1.h0.one.dim}
of Proposition~\ref{prop:resolv.dif} transforms into $h_1 = d_1 h_2$. Thus, by
Proposition~\ref{prop:resolv.dif} the corresponding resolvent difference is
one-dimensional if and only if $h_1 = d_1 h_2$.
\end{proof}
Proposition~\ref{prop:normal=const} shows that the class of normal operators
generated by BVP for equation~\eqref{eq:normal.2x2} is disappointedly small.
This result together with Theorem~\ref{th:one.dim.perturb.normal} makes it
reasonable to pose more general version of \textbf{Problem 1}. Namely we consider
\textbf{Problem 2} just replacing in formulation of \textbf{Problem 1} a normal
operator $T$ by an operator similar either to a normal operator or to an almost
normal operator.
%
%
\begin{theorem} \label{th:one.dim.perturb.similar_to_normal}
Let $b_1 b_2^{-1} \notin \bR$, and let $T := L_{U_1, U_2}(B, Q)$ and $\wt{S} :=
L_{\wt{U}_1,\wt{U}_2}(B, Q)$. Assume also that $Q_{21}(\cdot)$ admits a
holomorphic continuation to an entire function and $Q_{12} \equiv 0$, i.e.
condition \eqref{eq:Q12=0} is satisfied with $a = 0$. Let, finally, boundary
conditions $\{U_1,U_2\}$, $\{\wt{U}_1,\wt{U}_2\}$, be given by
\begin{align}
\label{eq:U1.U2.d1.d2.riesz}
 U_1(y) &= y_1(0) - d_1 y_1(1) = 0, \qquad
 U_2(y) = y_2(0) - d_2 y_2(1) = 0, \\
\label{eq:wt.U1.U2.d1.h1.riesz}
 \wt{U}_1(y) &= y_1(0) - h_1 y_2(0) = 0, \qquad
 \wt{U}_2(y) = y_1(1) - h_2 y_2(0) = 0,
\end{align}
where $d_1, d_2, h_1, h_2 \ne 0$. Then:
\begin{itemize}
\item[(i)] The operator $T$ is similar to an almost normal operator. In
particular, each eigenvalue of $T$ but finitely many, is algebraically simple
and the system of root vectors of $T$ forms a Riesz basis.
\item[(ii)] Assume in addition that $\|Q_{21}\|_{C[0,1]}$ is sufficiently small
and the algebraic condition~\eqref{eq:aj.ne.bc} holds. Then the operator $T$ is
similar to a normal operator. In particular, the system of its eigenvectors
forms a Riesz basis.

\item[(iii)] The operator $\wt{S}$ is peculiarly complete in $L^2([0,1]; \bC^2)$.

\item[(iv)] Resolvent difference $(\wt{S} - \l)^{-1} - (T - \l)^{-1}$ is
one-dimensional if and only if $h_1 = d_1 h_2$.
\end{itemize}
\end{theorem}
%
%
\begin{proof}
(i) This statement is immediate from Theorem \ref{th:similarity-to_normal}(i).

(ii) By Theorem~\ref{th:similarity-to_normal},~(ii), there exists $C > 0$ such
that the operator $L$ is similar to a normal operator whenever $\|Q\|_{C[0,1]}
= \|Q_{21}\|_{C[0,1]}$ is sufficiently small and condition~\eqref{eq:aj.ne.bc}
holds.

(iii)-(iv) These statements are proved in just the same way as
in Theorem~\ref{th:one.dim.perturb.normal}.
\end{proof}
Finally, we consider \textbf{Problem 3}, the most general version of
\textbf{Problem 1}. Namely, \textbf{Problem 3} is obtained form \textbf{Problem
1} by replacing the normality of $T$ by the property of its roots vectors to
constitute a Riesz basis with parentheses in $\fH$. It happen that for the dual
pair $\{L_{\min}(B,Q), L_{\min}(B^*,Q^*)\}$ \textbf{Problem 3} has an
affirmative solution for much wider class of potential matrices $Q(\cdot)$ than
in both previous cases.
%
%
\begin{theorem} \label{th:one.dim.perturb.riesz}
Let $T := L_{U_1, U_2}(B, Q)$ and $\wt{S} := L_{\wt{U}_1,\wt{U}_2}(B, Q)$. Let
in addition, $Q \in L^{\infty}([0,1]; \bC^{2 \times 2})$, $Q_{12}(\cdot)$ \
vanishes at a neighborhood of the endpoint 1, i.e. condition \eqref{eq:Q12=0}
holds, and let boundary conditions $\{U_1,U_2\}$, $\{\wt{U}_1,\wt{U}_2\}$ be
given by \eqref{eq:U1.U2.d1.d2.riesz}--\eqref{eq:wt.U1.U2.d1.h1.riesz} with
$d_1 d_2 h_1 h_2 \ne 0$. Then:
\begin{itemize}
\item[(i)] System of root vectors of the operator $T$ forms a Riesz basis with
parentheses in $L^2([0,1]; \bC^2)$.
\item[(ii)] Operator $\wt{S}$ is peculiarly complete in $L^2\([0,1];\bC^2\)$.
\item[(iii)] Resolvent difference $(\wt{S} - \l)^{-1} - (T - \l)^{-1}$ is
one-dimensional if and only if $h_1 = d_1 h_2$.
\end{itemize}
\end{theorem}
%
%
\begin{proof}
Statement (i) is immediate from Proposition~\ref{prop:basis.per}. Other
statements are proved in the same way as in Theorem~\ref{th:one.dim.perturb.normal}.
\end{proof}
Finally, we illustrate main results by considering the following example which
was our first initial observation while studying this problem (see in this
connection also~\cite{BarYak16}).
%
%
\begin{example}
Consider equation~\eqref{eq:normal.2x2} with $Q=0$. Setting $d_1 = d_2 = -1$
in~\eqref{eq:U1.U2.d1.d2} we arrive at antiperiodic boundary conditions
\begin{equation}\label{eq:normal.antiper}
 \wt U_1(y) = y_1(0) + y_1(1) = 0, \qquad
 \wt U_2(y) = y_2(0) + y_2(1) = 0.
\end{equation}
Denote by $L_{\rm{ap}}$ the operator generated in $L^2([0,1];\bC^2)$ by the
boundary value problem~\eqref{eq:normal.2x2},~\eqref{eq:normal.BQ},
\eqref{eq:normal.antiper}.
Assuming that $h_2 \ne h_1$ one easily finds inverse operators
$L^{-1}_{U_1,U_2}(\mathbf 0)$ and $L^{-1}_{\rm{ap}}(\mathbf 0)$:
\begin{equation}\label{8}
 L^{-1}_{U_1,U_2}(\mathbf 0)f = \begin{pmatrix} y_1(x) \\ y_2(x) \end{pmatrix}
 = \begin{pmatrix}
 i b_1 \left[\int^x_0 f_1(t)dt+\frac{h_1}{h_2-h_1}\int^1_0 f_1(t)dt\right] \\
 ib_2\int^x_0 f_2(t)dt + \frac{ib_1}{h_2-h_1}\int^1_0 f_1(t)dt
 \end{pmatrix}
\end{equation}
and
\begin{equation}\label{9}
 L^{-1}_{\rm{ap}}(\mathbf 0) f = \begin{pmatrix} y_1(x) \\ y_2(x) \end{pmatrix}
 = \begin{pmatrix}
 i b_1 \int^x_0 f_1(t)dt - i\frac{b_1}{2}\int^1_0 f_1(t)dt \\
 i b_2\int^x_0 f_2(t)dt - i\frac{b_2}{2}\int^1_0 f_2(t)dt
 \end{pmatrix}\,,
\end{equation}
Further, let
\begin{equation}
e_1=\binom{1}{0}\mathbf 1, \quad e_2 = \binom{0}{1}\mathbf 1 \in
L^2([0,1];\bC^2).
\end{equation}
Combining relation~\eqref{8} with~\eqref{9} one gets
\begin{equation}
 \left(L^{-1}_{U_1,U_2}(\mathbf 0) - L^{-1}_{\rm{ap}}(\mathbf 0) \right)f
 = \frac{ib_1}{h_2-h_1}(f,e_1)
 \begin{pmatrix} 2^{-1}(h_1+h_2) \\ 1 \end{pmatrix}
 + \frac{ib_2}{2}(f,e_2)\binom{0}{1}.
\end{equation}
Thus, $\rank(L^{-1}_{U_1,U_2}(\mathbf 0) - L^{-1}_{\rm{ap}}(\mathbf 0)) \le 2$ and,
if $h_2 + h_1 =0$, then the resolvent difference $L^{-1}_{U_1,U_2}(\mathbf 0) -
L^{-1}_{\rm{ap}}(\mathbf 0)$ is one-dimensional. Moreover, the operator $L_{U_1,
U_2}(\mathbf 0)$ is complete while its adjoint $L_{U_1,U_2}(\mathbf 0)^*$ is not
and the codimension of the span of its root functions is infinite.
\end{example}
%
%
\begin{remark}
Note that the authors of~\cite{BarYak16} (see also~\cite{BarYak15}) investigated
in great detail the completeness property of one-dimensional \textbf{non-weak
perturbations} of a compact self-adjoint operator $A$ in $\fH$. They obtain new
criteria for completeness of rank one (non-dissipative) perturbations, for joint
completeness of the operator and its adjoint, as well as for the spectral synthesis.
\end{remark}
%
%
\begin{remark}
Note, that \emph{non-degenerate separated} boundary conditions are always
\emph{strictly regular}, hence the root vectors of the corresponding BVP
constitute a \emph{Riesz basis}~\cite{DjaMit12UncDir}. Earlier these results
were proved for Dirac operator with $Q \in L^2([0,1]; \bC^{2 \times 2})$ by
P.~Djakov and B.~Mityagin~\cite{DjaMit12UncDir} (see also~\cite{Bask11}).
\end{remark}
%
%
\begin{example}\label{ex:Dirac_with_Q=0}
Let us briefly discuss BVPs for Dirac type systems.

(i) First we consider BVPs~\eqref{eq:system.nxn}--\eqref{eq:BC.nxn} with a
nonsingular $n\times n$ diagonal matrix $B = \diag(b_1, b_2, \ldots, b_n) = B^*$
and $Q=0$. It is shown in~\cite{MalOri12} that in this case the operator
$L_{C,D}(B,0)$ \emph{is complete if and only if} BC~\eqref{eq:BC.nxn} are regular
$(\Longleftrightarrow \ \text{weakly regular})$, i.e. the conditions~\eqref{2.4Intro}
are satisfied. Therefore, by Theorem \ref{th2.1_MO}, the operators $L_{C,D}(B,0)$
and $L_{C,D}(B,0)^*$ are complete only simultaneously.

This example, as well Example~\ref{Ex_S-L_oper_Intro}, gives a negative solution
to \textbf{Problem 1} for simplest Dirac-type and Sturm-Liouville operators,
respectively. Both examples are opposite to the one given by the
BVP~\eqref{eq:system}--\eqref{eq:BC} with $Q=0$ and the diagonal matrix $B$
satisfying $b_1 b_2^{-1} \not \in \bR$.

(ii) Next we consider $2\times 2$ Dirac type system~\eqref{eq:normal.2x2} with
\begin{equation}\label{eq:Dirac.BQ}
 B = \text{diag}(b_1,b_2)= B^*, \quad \text{and}\quad
 Q = \begin{pmatrix} 0 & Q_{12} \\ Q_{21}& 0 \end{pmatrix}
 \in L^1([0,1];\bC^{2\times 2}).
\end{equation}
assuming that $b_1 b_2 < 0$. In the case $Q(\cdot) \ne 0$ several sufficient
conditions of completeness of non-regular (and even degenerate)
BVPs~\eqref{eq:system}--\eqref{eq:BC} (operators $L_{U_1,U_2}(B,Q)$) were obtained
in~\cite{MalOri12},~\cite{LunMal13Dokl}. An interesting feature of these results
is that they ensure completeness of both operators
$L_{U_1,U_2}(B,Q)$ and $(L_{U_1,U_2}(B,Q))^*$.

In connection with \textbf{Problem 1} it is interesting to examine the
BVP~\eqref{eq:normal.2x2},~\eqref{eq:U1.U2.h0.h1.again}. Since $J_{14} = J_{24}
= 0$ and $J_{32} J_{13} = h_1 \cdot 1 \ne 0,$ boundary conditions
\eqref{eq:U1.U2.h0.h1.again} meet all the assumptions of Theorem~5(i) from
\cite{LunMal13Dokl} but one. Namely, Theorem~5(i) from~\cite{LunMal13Dokl}
ensures completeness of the problem~\eqref{eq:normal.2x2},
\eqref{eq:U1.U2.h0.h1.again} and its adjoint provided that $Q\in W^{k,2}([0,1];
\bC^{2 \times 2})$ and $Q_{12}^{(j)}(1) \ne 0$ for some $j \in \{0,\ldots, k-1\}$.
It happens, in particular, if either $Q$ is analytic at the endpoint $1$ and
$Q(\cdot)\not \equiv 0$, or $Q\in C^{\infty}([1-\eps,1]; \bC^{2 \times 2})$ and
$Q_{12}^{(j-1)}(1) \ne 0$ for some $j \in \bN$. In all these cases the
\textbf{Problem 1} with $Q=Q^*$ \emph{has a negative solution}.

However, it remains open for the operator $L$ generated by the
problem~\eqref{eq:normal.2x2},~\eqref{eq:U1.U2.h0.h1.again} with
non-analytic $Q = Q^* \in C^{\infty}([1-\eps, 1]; \bC^{2 \times 2})$ and
satisfying $Q_{12}^{(j-1)}(1) = 0$ for all $j \in \bN$.
\end{example}
%
%
\begin{example} \label{ex:Sturm-Liouville}
Consider Sturm-Liouville equation
\begin{equation}\label{eq:St-Liouv_equation}
 L(q)y := -y'' + q(x)y = \l y, \qquad q\in L^1[0,1],
\end{equation}
subject to general linear boundary conditions
\begin{equation}\label{eq:BC_for_St-Liouv}
	U_j(y) := a_{j 1}y(0) + a_{j 2}y'(0) + a_{j 3}y(1) + a_{j 4}y'(1)= 0,
	\qquad j \in \{1,2\},
\end{equation}
where the linear forms $\{U_j\}_{j=1}^2$ are assumed to be linearly independent.
Denote by $L_{U_1, U_2}(q)$ the operator associated in $L^2[0,1]$ with the
BVP~\eqref{eq:St-Liouv_equation}--\eqref{eq:BC_for_St-Liouv}.

Recall that BC for Sturm-Liouville equation are called nondegenerate if the
characteristic determinant $\Delta(\cdot)$ of the BVP is not reduced to a
constant, $\Delta \ne \const$. It is well known (see e.g.~\cite[Theorem~1.3.1]{Mar86})
that the BVP~\eqref{eq:St-Liouv_equation}--\eqref{eq:BC_for_St-Liouv} (the
operator $L_{U_1, U_2}(q)$) is complete whenever BC are nondegenerate. In this
case both operators $L_{U_1, U_2}(q)$ and $(L_{U_1, U_2}(q))^*$ are complete
simultaneously.

Passing to degenerate BC we note they are equivalent (see e.g.~\cite{Mal08}) to
a pair of conditions of the following one-parameter family
\begin{equation} \label{eq:Degen_BC_for_St-Liouv}
 U_{1,\alpha}(y) := y(0) - \alpha y(1), \quad
 U_{2,\alpha}(y) := y'(0) + \alpha y'(1), \qquad \alpha\in \bC\setminus\{0\}.
\end{equation}
We put $L_{\alpha}(q) := L_{U_{1,\alpha}, U_{2,\alpha}}(q)$ and note that the
adjoint operator $L_{\alpha}(q)^* := (L_{\alpha}(q))^* = L_{\beta}(\ol{q})$, i.e.
$L_{\alpha}(q)^*$ is given by expression~\eqref{eq:St-Liouv_equation} with
$\ol{q}$ instead of $q$ and the BC $U_{1,\beta}$ and $U_{2,\beta}$ of the form~\eqref{eq:Degen_BC_for_St-Liouv} with $\beta = - 1/\ol{\alpha}$.

Completeness property for BVP~\eqref{eq:St-Liouv_equation}--\eqref{eq:BC_for_St-Liouv}
with degenerate BC was investigated in~\cite{Mal08} and~\cite{Mak14}. All
known sufficient conditions ensure completeness of operators $L_{q,\alpha}(q)$
for all $\alpha\in \bC\setminus\{0\}$ and, in particular, completeness of
$L_{\alpha}(q)^* = L_{\beta}(\ol{q})$. For instance, a result
from~\cite{Mal08} guaranties completeness of these operators whenever
$q_{\rm{odd}}(x) := q(x) - q(1 - x)$ is smooth and $q_{\rm{odd}}^{(k-1)}(0)
\ne 0$ for some $k \in \bN$. However, the following problem remains open:

Is there exist a non-analytic potential $q = \ol{q} \in C^{\infty}[0,1]$
satisfying $q_{\rm{odd}}^{(j-1)}(0) = 0$ for all $j \in \bN$ and such that the
operator $L_{q,\alpha}$ is peculiarly complete?

The existence of such a real potential $q$ would lead to a positive solution to
\textbf{Problem 1} for Sturm-Liouville operator~\eqref{eq:St-Liouv_equation}
with such $q$. However, we conjecture that the answer is negative.
\end{example}
%
%
\section{Appendix. Regular and weakly regular boundary value problems}
\label{sec:appendix}
%
%
Let us recall the definition of regular boundary conditions from~\cite[{p.89}]{BirLan23}.
We use the following construction. Let $A=\diag(a_1, \ldots, a_n)$ be a diagonal
matrix with entries $a_k$ (not necessarily distinct) that are not lying on the
imaginary axis, $\Re a_k \ne 0$. Starting from arbitrary matrices $C, D \in
\bC^{n \times n}$, we define the auxiliary {$n \times n$} matrix $T_A(C,D)$ as
follows:
\begin{itemize}
\item if $\Re a_k>0$, then the $k$th column in the matrix $T_A(C,D)$ coincides
 with the $k$th column of the matrix $C$,
\item if $\Re a_k<0$, then the $k$th column in the matrix $T_A(C,D)$ coincides
 with the $k$th column of the matrix $D$.
\end{itemize}
It is clear that $T_A(C,D)=T_{-A}(D,C)$.
%
%
\begin{definition} \label{def1.1}
The boundary conditions~\eqref{eq:BC.nxn} are called \emph{regular} whenever
$\det T_{izB}(C,D) \ne 0$ for every \emph{admissible} $z \in \bC,$ i.e. for such
$z$ that $\Re(izB)$ is nonsingular.
\end{definition}
%
%
To understand this definition better consider the lines $\{\l \in \bC : \Re(i
b_j \l)=0\}$, $j \in \{1,2,\ldots,n\}$, of the complex plane. They divide the
complex plane in {$m = 2m' \le 2n$} sectors. Denote these sectors by $\sigma_1,
\sigma_2, \ldots \sigma_m$. Let $z_1, z_2, \ldots, z_m$ be complex numbers such
that $z_j$ lies in the interior of $\sigma_j, j \in \{1,\ldots, m\}$. The
boundary conditions~\eqref{eq:BC.nxn} are regular whenever
\begin{equation}\label{regular}
 \det T_{i z_j B}(C,D) \ne 0 , \qquad j \in \{1,\ldots, m\}.
\end{equation}

Let us recall the concept of \emph{weakly regular} boundary conditions
from~\cite{MalOri12} and completeness results for BVP with such conditions.
%
%
\begin{definition}[\cite{MalOri12}] \label{def:weakly.regular}
The boundary conditions~\eqref{eq:BC.nxn} are called \emph{weakly $B$-regular}
(or, simply, weakly regular) if there exist three complex numbers $\{z_j\}^3_1$
satisfying the following conditions:

(a) the origin is an interior point of the triangle $\triangle_{z_1z_2z_3};$

(b) $\det \, T_{i z_j B}(C,D) \ne 0$ for $j \in \{1,2,3\}$.
\end{definition}
%
%
\noindent The following result is contained
in~\cite[Theorem 1.2 and Corollary 3.3]{MalOri12}.
%
\begin{theorem} \label{th2.1_MO}
Let $Q \in L^1\([0,1]; \bC^{n\times n}\)$ and let boundary conditions
\eqref{eq:BC.nxn} be weakly $B$-regular. Then the system of root functions of
the BVP~\eqref{eq:system.nxn}--\eqref{eq:BC.nxn} (of the operator $L_{C,D}(Q)$)
is complete and minimal in $L^2\([0,1]; \bC^n\)$.

Moreover, the system of root functions of the adjoint operator $L_{C,D}(Q)^*$
is also complete and minimal in $L^2\([0,1]; \bC^n\)$.
\end{theorem}
%
%
\begin{corollary} \label{cor2.1}
Let $Q \in L^1\([0,1]; \bC^{n \times n}\)$ and let the matrices $T_{izB}(C,D)$ and
$T_{-izB}(C,D) = T_{izB}(D,C)$ be nonsingular for some $z \in \bC$. Then

(i) The boundary conditions~\eqref{eq:system} are weakly $B$-regular.

(ii) The system of EAF of the operator $L_{C,D}(Q)$ is complete and minimal in
$L^2\([0,1];\bC^n\)$.
\end{corollary}
%
%
For $n\times n$ Dirac type system ($B = B^*$) the concept of {weakly regular}
BC~\eqref{eq:BC.nxn} coincides with that of regular ones and reads as follows
\begin{equation}\label{2.4Intro}
 \det(CP_{+} +\ DP_{-}) \ne 0 \quad\text{and} \quad \det(CP_{-} +\ DP_{+}) \ne 0.
\end{equation}
Here $P_+$ and $P_-$ are the spectral {projections} onto "positive"\ and
"negative"\ parts of the spectrum of $B=B^*$, respectively. Hence, by
Theorem~\ref{th2.1_MO}, under conditions~\eqref{2.4Intro} \emph{both operators
$L_{C,D}(Q)$ and $L_{C,D}(Q)^*$ are complete and minimal}. In the case $n=2$
condition~\eqref{2.4Intro} turns into $J_{14} J_{32} \ne 0$.

Consider system~\eqref{eq:system} with the matrix $B=\diag (b_1^{-1}, b_2^{-1})
\ne B^*$ assuming that $b_1/b_2 \notin \bR$. In this case the lines $\{\l \in \bC:
\Re(i b_j \l)=0\}$, $j \in \{1,2\}$, divide the complex plane in two pairs of
vertical sectors and Corollary~\ref{cor2.1} guarantees the completeness and the
minimality of the root system of problem~\eqref{eq:system}--\eqref{eq:BC} in the
following cases:
\begin{equation}\label{cond:w-reg}
 \text{(i)} \quad J_{14}J_{32} \not = 0\quad \text{or}\quad
 \text{(ii)} \quad J_{12}J_{34}\not =0.
\end{equation}
Note that the regularity of boundary conditions~\eqref{eq:BC.nxn} implies in
particular that
\begin{equation*}
 J_{14}J_{32}J_{12}J_{34}\ne 0.
\end{equation*}
\textbf{Acknowledgement.} The publication was prepared with the support of the
RUDN University Program 5-100.
%
%

%
%
\end{document}